\documentclass[12pt, reqno]{amsart}

\title[Analysis of Semilinear Heat Equations with Hardy Potential]{Local and Global Analysis of Semilinear Heat Equations with Hardy Potential on Stratified Lie Groups}

\emergencystretch=\maxdimen
\author[D. Suragan and B. Talwar]{Durvudkhan Suragan and Bharat Talwar}
\address{Durvudkhan Suragan, Department of Mathematics, Nazarbayev University, Astana 010000, Kazakhstan}
\email{durvudkhan.suragan@nu.edu.kz}
\address{Bharat Talwar, Department of Mathematics, Indian Institute of Science Education and Research, Bhopal 462066, Madhya Pradesh, India}
\email{btalwar.math@gmail.com, bharattalwar@iiserb.ac.in}

\allowdisplaybreaks[1]

\usepackage{relsize}
\usepackage{amsmath,amsfonts,amssymb,amsthm,mathrsfs}
\usepackage{hyperref}
\usepackage{cleveref}
\usepackage[margin=2.5 cm]{geometry}
\usepackage{setspace}
\usepackage{mathtools}  

\usepackage{setspace}
\usepackage{xcolor}

\numberwithin{equation}{section}
\newtheorem{theorem}{\bf Theorem}[section]
\newtheorem{lemma}[theorem]{\bf Lemma}

\newtheorem{remark}[theorem]{\bf Remark}

\newtheorem{defin}[theorem]{\bf Definition}

\newcommand{\seq}{\subseteq}
\newcommand{\G}{\mathbb{G}}

\newcommand{\N}{\mathbb{N}}

\newcommand{\R}{\mathbb{R}}

\newcommand{\ol}{\overline}

\makeatletter \@namedef{subjclassname@2020}{\textup{2020} Mathematics Subject Classification} \makeatother

\begin{document}
	\begin{abstract}
On stratified Lie groups we study a semilinear heat equation with the Hardy potential, a power non-linearity and a forcing term which depends only upon the spacial variable.
The analysis of an equivalent formulation to the problem and an application of a decade old result of Avelin et al. facilitates the management of the singularity in the Hardy potential, thereby yielding results pertaining to both  local and global nonexistence.
In addition, local existence is verified when the gradient term appearing in the Hardy potential is unimodular almost everywhere.
The global existence is also proved under an additional assumption that the forcing term depends on the time variable as well.
Through these results this paper sheds light on the  possible pivotal exponents  for the existence of both local and global solutions to the equation, offering a deeper understanding of the interplay between the model's parameters and the underlying stratified Lie group structure. 
	\end{abstract}
	\maketitle
	
	\section{Introduction}
H\"{o}rmander's operators are being  studied actively for more than half a century now.
Various partial differential equations, for instance- Kolmogorov-Fokker-Planck equations and the equations determining Brownian motion in periodic potentials, may be written using such operators \cite{BaramantiBook}.
	The so-called lifting technique from \cite{RothschildStein} enables in simplifying the study of these operators by focusing on the analysis of a hypoelliptic operator, the sub-Laplacian.
	In this technique, one locally approximates a H\"{o}rmander's sum of square operator with a sub-Laplacian on a higher dimensional free stratified Lie group.
	Ever since the inception of this technique, the study of partial differential equations on stratified Lie groups have become an active area of research.

Let $\G$ be a stratified Lie group with homogeneous dimension $Q$.
Assume that $Q \geq 3$, so the classical cases of the abelian groups $\R$ and $\R^2$ are omitted.
For the $\Delta_\G$-gauge $d$, except in \Cref{Conjecture} where $d$ is assumed to be a cc-distance, we study the impact of the Hardy potential  $\frac{\left| \nabla_\G d(x) \right|^2}{d(x)^2}$ on the problem we studied in \cite{SuraganTalwar2}. A similar study involving the Hardy potential is made for parabolic equation recently \cite{AbdellaouiSiclariPrimo}.
More precisely, for the self-adjoint and hypoelliptic sum of squares operator $$\Delta_\G = \sum_{i=1}^m X_i^2$$ and $p > 1$, we consider the semilinear subelliptic problem
	\begin{equation}\label{MainProblem}
		\begin{cases} 
			u_t(t,x) - \Delta_\G u(t,x) - \frac{\lambda \left| \nabla_\G d(x) \right|^2 u}{d(x)^2} =  u(t,x)^p + f( x), & (t,x) \in (0,T) \times \G, \\
			u(0, x) = u_0(x), & x \in \G,
		\end{cases}
	\end{equation}
	where $\lambda > 0$ and $f,u \geq 0$.
	Our aim is to find those conditions on the scalar $\lambda$ and the exponent $p$ which guarantee the existence and non-existence of, both, local and global in time weak solutions.
	For this, we derive an equivalent problem (see \ref{MainProblem2}) posed on appropriate weighted spaces and then use the results available with this new setting.
	This approach provides insight on why weighted blow up is observed for weak solutions of certain semilinear heat equations with Hardy potential.
	
It turns out that including the Hardy potential in the problem takes us from $L^p$ space to weighted $L^p$ space.
On the same note, let us now emphasize a similar phenomenon.
Recall that a critical exponent $\alpha$ is the positive real number such that a global solution never exists if $p \leq \alpha$ while a global solution exists for some $u_0$ and $f$ when $\alpha < p$.
When $\G = (\R^n, +)$ and there is neither the potential term $\frac{\left| \nabla_\G d(x) \right|^2}{d(x)^2}$ nor the forcing term $f(x)$ in the problem, the critical exponent is  $\frac{n+2}{n}$ \cite{Hayakawa, Fujita}.
However, when one includes a non-negative forcing term $f(x)$ in this problem, the critical exponent changes to $\frac{n}{n-2}$ \cite{Bandle}.
So, the effect of including the term $f(x)$ is that it takes us from $\R^n$ to $\R^{n-2}$ in the sense that the critical exponent for the problem
\begin{equation*}
\begin{cases} 
u_t(t,x) - \Delta u(t,x) =  u(t,x)^p, & (t,x) \in (0,T) \times \R^n, \\
u(0, x) = u_0(x), & x \in \R^n,
\end{cases}
\end{equation*} is the same as that of
\begin{equation*}
\begin{cases} 
u_t(t,x) - \Delta u(t,x) =  u(t,x)^p + f(x), & (t,x) \in (0,T) \times \R^{n+2}, \\
u(0, x) = u_0(x), & x \in \R^{n+2}.
\end{cases}
\end{equation*}

Several of our findings are natural continuation of the classical results which appeared in \cite{Ishige, AbdellaouiAlonsoRamos}.
In \cite{AbdellaouiAlonsoRamos},  the problem  (\ref{MainProblem}) was studied on $\G = (\R^n, +)$ with the same potential  $\frac{1}{|x|^2}$, since $|\nabla_\G d(x)| \equiv 1$ in this case.
However, the effect of the forcing term $f$ was not taken into account, resulting in $\frac{n-\alpha+2}{n-\alpha}$ as the critical exponent for global existence of a solution for some $\alpha \geq 0$.
When $\G$ is the Heisenberg group $\mathbb{H}^n$, the corresponding linear case of (\ref{MainProblem}) was recently considered in \cite{GoldsteinGoldsteinKogojRhandiTacelli} where the instantaneous blow-up results from \cite{BarasGoldstein, GoldsteinZhang} on $\G = (\R^n, +)$  were generalized to $\mathbb{H}^n$.
With $d(z,l) = (|z|^4 + l^2)^{1/4}$ for $(z,l) \in \mathbb{H}^n$, a critical juncture in their proofs was the fact that one can derive an explicit expression for $\left| \nabla_{\mathbb{H}^n} d(x) \right|$.
	This becomes a major challenge for us as such an expression may not be derived for an arbitrary stratified Lie group.
	So, we take an alternative approach and in the aftermath obtain results which were unknown even for the Heisenberg group.
	Nevertheless, the critical constant  $\frac{Q-2}{2}$ for the Hardy inequality on stratified Lie groups plays a crucial role in our proofs as well.
	Besides the aforementioned articles, the reader may refer to \cite{Gu} and references therein for further related  studies.
	
	It follows from the comparison theorem that a  positive solution $u$ for (\ref{MainProblem}) is bounded from below by a positive solution $v$ for \begin{equation}\label{GWithForcingTerm}
		\begin{cases} 
			v_t(t,x) - \Delta_\G v(t,x) = v(t,x)^p + f( x), & (t,x) \in (0,T) \times \G, \\
			v(0, x) = u_0(x), & x \in \G.
		\end{cases}
	\end{equation}
	Using the results of \cite{SuraganTalwar2} one may now conclude  that there is no global in time weak solution for (\ref{MainProblem}) when $p \leq \frac{Q}{Q-2}$.
	In fact, we will prove that an optimal exponent for non-existence of solutions must be strictly larger than $\frac{Q}{Q-2}$.
	What is more interesting is that, unlike the case of (\ref{GWithForcingTerm}), a local solution may not always exist for (\ref{MainProblem}).

The detailed plan of this paper is as follows.
In \Cref{Preliminaries} we start by supplying the necessary background and definitions for smooth reading of this paper.
Equivalent formulations for (\ref{MainProblem}) are then provided.
This approach and \cite{AvelinCapognaCittiNysrtom} helps us prove that a solution must satisfy a qualitative and quantitative estimate near the identity of $\G$ and that it must belong to certain weighted space.
Deriving this result from \cite{AvelinCapognaCittiNysrtom} is instrumental in our study as this makes it possible for us to apply the classical techniques of \cite{AbdellaouiAlonsoRamos}.
	It is also shown that a weak solution does not exist for $\lambda > \left(\frac{Q-2}{2}\right)^2$, so we can always focus on the other case.
	For this we define $\alpha$ to be the smaller of the two roots of the polynomial $\alpha^2 - (Q-2)\alpha + \lambda$.
	In \Cref{LocalExistence}, we  show that a local weak solution does not exist if $p \geq 1+ \frac{2}{\alpha}$.
	It is also proved that for $p \leq \frac{Q-\alpha}{Q-2-\alpha}$ a global solution $u \in L^1_{loc, d^{-\alpha}}(\G)$ does not exist.
	Finally, in \Cref{Conjecture} we conjecture that $1+ \frac{2}{\alpha}$ and $\frac{Q-\alpha}{Q-2-\alpha}$ are the critical exponents for local and global existence of a weak solution for (\ref{MainProblem}), respectively.
	To support these hypothesis, we  make the study of two separate cases.
	For the local existence, we consider the case when $|\nabla_\G d| = 1$ almost everywhere -see \cite{MontiSerra}.
	Whereas for global existence  we consider a variation of (\ref{MainProblem}) with $|\nabla_\G d| = 1$ almost everywhere and $f$ dependent on both time and space variables.

\section{Preliminaries}\label{Preliminaries}

For the $\Delta_\G$-gauge $d$, the useful Hardy inequality on stratified Lie  groups \cite{Suragan, GoldsteinKombe, Ambrosio}, is
$$\left\| \frac{u}{d(x)} \right\|_2 \leq \frac{2}{Q-2}\left\| \nabla_\G u \right\|_2, \  \ \forall \ u \in C^\infty_c(\G).$$
By replacing $u$ with the function $u_1(x) := u(2 x)$, it is verified that  if $$\left\| \frac{u}{d(x)^s} \right\|_2 \leq C \left\| \nabla_\G u \right\|_2$$ for some constant $C$, then $s = 1$.
Moreover, the best constant $\frac{Q-2}{2}$ is never attained in the space $S_0^1(\G)$, which  the completion of $C^\infty_c(\G)$ with respect to the norm given by $\left\| u \right\|^2 = \int_{\G} (| \nabla_\G u|^2 + |u|^2)$.

For a reason which will become evident in \Cref{IntegralOfDpowerMinusAlphauPowerpIsFinite} (2), let us fix $$\left(\frac{Q-2}{2}\right)^2 \geq \lambda > 0.$$	
For such $\lambda$, set $$\alpha:= \alpha_{-}(\lambda) = \frac{Q-2}{2} - \sqrt{\left(\frac{Q-2}{2}\right)^2 - \lambda}$$ to be the smaller of the two zeroes of the polynomaial $\alpha^2 - (Q-2)\alpha + \lambda$.
Similarly, let $$\alpha_{+}(\lambda) = \frac{Q-2}{2} + \sqrt{\left(\frac{Q-2}{2}\right)^2 - \lambda}.$$
Using \cite[Proposition 5.4.3]{Bonfiglioli} along with the fact that for any $h \in C^2(0,\infty)$ we have $\Delta_\G (h \circ d) = h''(d) |\nabla_\G d|^2 + h'(d) \Delta_\G d$, one obtains $$\frac{ -\lambda \left| \nabla_\G d(x) \right|^2}{d(x)^2}  = d(x)^\alpha \Delta_\G (d^{- \alpha}(x)).$$
Thus, the problem (\ref{MainProblem})  is equivalent to the problem \begin{equation}\label{MainProblem1}
		\begin{cases}
			u_t(t,x) - \Delta_\G u(t,x) + d(x)^\alpha \Delta_\G (d^{- \alpha}(x))u =  u(t,x)^p + f(x), & (t,x) \in (0,T) \times \G, \\
			u(0, x) = u_0(x), & x \in \G,
		\end{cases}
	\end{equation}
where $f,u \geq 0$.
	This suggests that the function on $\G \setminus \{ 0\}$ given by $$x \to d(x)^{-\alpha}$$ has something to do with the solutions of (\ref{MainProblem}).
	
	To dig deeper into this idea, let $u  = d^{-\alpha} w$.
	Then,
	\begin{eqnarray*}
		u^p + f(x)&=& u_t(t,x) - \Delta_\G u(t,x) + d(x)^\alpha \Delta_\G (d^{- \alpha}(x))u \\
		&=& d^{-\alpha}(x) w_t(t,x) - \Delta_\G \left( d^{-\alpha}(x) w(t,x) \right) - \frac{\lambda \left| \nabla_\G d(x) \right|^2 d^{-\alpha} w}{d(x)^2} \\
		&=& d(x)^{-\alpha} w_t(t,x) - (\Delta_\G d(x)^{-\alpha} ) w(t,x) -  (\Delta_\G w(t,x) ) d(x)^{-\alpha} \\ 
		&-& 2(\nabla_\G w(t,x) )  (\nabla_\G d(x)^{-\alpha} )  - \frac{\lambda \left| \nabla_\G d(x) \right|^2 w(t,x)}{d(x)^{2+\alpha}} \\
		&=& d(x)^{-\alpha} w_t(t,x) -   d(x)^{-\alpha} \Delta_\G w(t,x)  -  2(\nabla_\G w(t,x) )  (\nabla_\G d(x)^{-\alpha} ).
	\end{eqnarray*}
	Hence,
	\begin{eqnarray*}
		d^{-\alpha} \left(d^{-\alpha p}w^p + f(x) \right)
		&=& d^{-\alpha} \left( u^p + f(x) \right) \\
		&=& d^{-2\alpha} w_t -   d^{-2\alpha} \Delta_\G w -  2 d^{-\alpha}  (\nabla_\G w )  (\nabla_\G d^{-\alpha} )  \\
		&=&  d^{-2\alpha} w_t -   d^{-2\alpha} \Delta_\G w   -    (\nabla_\G w )  (\nabla_\G d^{-2\alpha} )  \\
		&=& d^{-2\alpha} w_t -   \nabla_\G( d^{-2\alpha} \nabla_\G w )
	\end{eqnarray*}
	So, the problem (\ref{MainProblem}) is in turn equivalent to the problem \begin{equation}\label{MainProblem2}
		\begin{cases} 
			d^{-2\alpha} w_t -   \nabla_\G( d^{-2\alpha} \nabla_\G w ) = d^{-\alpha(p+1)} w(t,x)^p +  d^{-\alpha}f(x), & (t,x) \in (0,T) \times \G, \\
			w(0, x) =  d^{\alpha} u_0(x), & x \in \G
		\end{cases}
	\end{equation}
	where $f,w \geq 0$.

Under the condition that $\left(\frac{Q-2}{2}\right)^2 \geq \lambda > 0$, instead of studying (\ref{MainProblem}) directly, we focus on its equivalent form (\ref{MainProblem2}) with which we are more familiar.
For $r > 0$, define the open ball $$B_r(0)= \{ x \in \G: d(x) < r \}$$ and the punctured closed ball $$B_r'(0) = \{ x \in G :  0 < d(x) \leq r \}.$$
	We fix  $\frac{1}{p} + \frac{1}{p'} = 1$.
	As needed for the following definition to make sense, assume from now onward that $$\int_{\G} d^{-\alpha}(x) f(x) dx < \infty.$$
	Let us emphasize that the measure $dx$ must not be confused with the seminorm $d(x)$.
	\begin{defin}\label{DefinitionLocalWeakSolution}
		Let $w_0 \in L^1_{loc, d^{-2 \alpha}(x)}(\G)$.
		Then $$w \in L^1_{loc, d^{-2 \alpha}(x)}(\G) \cap  L^{p}_{loc, d^{\frac{-\alpha(p+1)}{p}}(x)}\left((0,T), L^{p}_{loc, d^{   \frac{-\alpha(p+1)}{p}  }(x)  }(\G)\right)$$ is called a local in time weak solution of (\ref{MainProblem2}) if for every non-negative test function $\psi \in C^1_c( (0,T); C_c(\G) )$ it satisfies
		\begin{eqnarray*}
			\int_0^T \int_{\G} \left(d^{-\alpha(p+1)}(x) w(t,x)^p \right. &+& \left. d^{-\alpha}(x)f(x)\right) \psi + \int_{\G} d^{-2\alpha}(x) w(0,x) \psi(0, x) \\
			&+&  \int_0^T \int_{\G} d^{-2\alpha}(x) w(t,x) \psi_t + \int_0^T \int_{\G} w \left(\nabla_\G (d^{-2\alpha}(x) \nabla_\G  \psi)\right) = 0.
		\end{eqnarray*}
		When $T = \infty$, such $w$ is called global in time weak solution.
	\end{defin}
	
	\begin{remark}\label{unboundedness}
		Let $u$ be a solution to (\ref{MainProblem}).
		Then $w(t,x):= u(t,x) d^{\alpha}(x)$ is a solution for (\ref{MainProblem2}).
		One may apply comparison theorem from \cite[Theorem 2.1]{RuzhanskySuraganBLMS} followed by the Harnack inequality from \cite[Theorem 1.1]{AvelinCapognaCittiNysrtom} to obtain $r, c(r)> 0$ and $0< t_1 < t_2 < T < \infty$ such that $u(t,x) \geq c(r) d^{-\alpha}(x)$ in $   [t_1, t_2] \times B_r'(0)$.
	\end{remark}
Following is a key lemma which might be available in the literature, but we were unable to trace it.
	\begin{lemma}\label{infiniteintegral}
		For every $r >0$ the integral $$\int\limits_{B_r(0)} d(x)^{-Q}  \left| \nabla_\G d(x) \right|^2$$ is unbounded.
	\end{lemma}
	\begin{proof}
		For $n \in \N$, let $K_n = \{ x \in B_r(0): \frac{r}{2^n} < d(x) \leq \frac{r}{2^{n-1}}  \}$.
		An application of Hardy's inequality gives
		\begin{eqnarray*}
			\int\limits_{B_r(0)} d(x)^{-Q}  \left| \nabla_\G d(x) \right|^2 dx
			&=& \sum_{n \in \N} \int_{K_n} d(x)^{-Q}  \left| \nabla_\G d(x) \right|^2 dx\\
			& \geq & \sum_{n \in \N} \left(\frac{2^{n-1}}{r}\right)^Q \int_{K_n}  \left| \nabla_\G d(x) \right|^2 dx\\
			& \geq & C \sum_{n \in \N} \left(\frac{2^{n-1}}{r}\right)^Q \int_{K_n} dx\\
			&=& C \sum_{n \in \N} \left(\frac{2^{n-1}}{r}\right)^Q  \left(  \left(\frac{r}{2^{n-1}}\right)^Q - \left(\frac{r}{2^n}\right)^Q \right)\\
			&=& C \sum_{n \in \N}  \left(  1 - \left(\frac{1}{2}\right)^Q \right)
		\end{eqnarray*}
		which diverges to infinity.
	\end{proof}
The subsequent remark presents a couple of non-existence results  while clarifying why it is sufficient to focus on the case $\lambda \leq \left(\frac{Q-2}{2}\right)^2$.
	\begin{remark}\label{IntegralOfDpowerMinusAlphauPowerpIsFinite}
		From what is already demonstrated, if $w$ is a weak solution for (\ref{MainProblem2}) then $u = d^{- \alpha} w$ is a weak solution for (\ref{MainProblem}).
		Thus, $u$ must necessarily satisfy $$u_0 \in  L^1_{loc, d^{-\alpha}(x)}(\G),$$ $$\int_0^T \int_{\G} d^{-\alpha}(x) \left( u(t,x)^p +  f(x)\right) < \infty$$
		and
		\begin{eqnarray*}
			\int_0^T \int_{\G} \left(d^{-\alpha}(x) u(t,x)^p \right. &+& \left. d^{-\alpha}(x)f(x)\right) \psi + \int_{\G} d^{-\alpha}(x) u_0 \psi(0, x) \\
			&+&  \int_0^T \int_{\G} d^{-\alpha}(x) u \psi_t + \int_0^T \int_{\G} d^{\alpha}(x) u \left(\nabla_\G (d^{-2\alpha}(x) \nabla_\G  \psi)\right) = 0.
		\end{eqnarray*}
We now consider the two crucial cases to begin with.
\begin{enumerate}
\item\label{NecessityForWeakSupersolution} If $1+ \frac{2}{\alpha} +2 \times \frac{\sqrt{\left(\frac{Q-2}{2}\right)^2 - \lambda}}{\alpha} \leq p$.

Let us suppose that $u$ is a weak solution to (\ref{MainProblem}).
Under the given hypothesis, we have $$p + 1  \geq 2 + \frac{2}{\alpha} +2 \left( \frac{\sqrt{\left(\frac{Q-2}{2}\right)^2 - \lambda}}{\alpha}\right)$$ and hence $$\alpha(p+1) \geq 2\alpha + 2 + 2 \left( \sqrt{\left(\frac{Q-2}{2}\right)^2 - \lambda}\right) =  Q.$$
			Using \Cref{unboundedness} we now obtain $$\int_{t_1}^{t_2} \int_{B_r(0)} d(x)^{-\alpha} u^p \geq C \int_{B_r(0)} d(x)^{-\alpha} d(x)^{-p \alpha}  \geq C \int_{B_r(0)} d(x)^{-Q} = \infty.$$
			This is a contradiction.
			Thus, (\ref{MainProblem}) does not admit a weak solution in this case.

\item\label{supercriticalforlambda} If $\lambda > \left(\frac{Q-2}{2}\right)^2$.
			
Let us suppose that $u$ is a weak solution for (\ref{MainProblem}) and present this problems as
\begin{equation*}
\begin{cases} 
u_t - \Delta_\G u - \frac{ \left(\frac{Q-2}{2}\right)^2 \left| \nabla_\G d \right|^2 u}{d^2} =  \frac{ \left(\lambda - \left(\frac{Q-2}{2}\right)^2 \right) \left| \nabla_\G d \right|^2 u}{d^2} +  u^p + f(x), &  \text{ on } (0,T) \times \G, \\
u(0, x) = u_0(x), & x \in \G.
\end{cases}
\end{equation*}
			From this frame of reference, where new $\lambda$ is $\left(\frac{Q-2}{2}\right)^2$ and $\alpha$ is $\frac{Q-2}{2}$, we obtain
			$$\int_{t_1}^{t_2} \int\limits_{B_r(0)} d(x)^{-\frac{Q-2}{2}}  \left(   \frac{\left(\lambda - \left(\frac{Q-2}{2}\right)^2\right) \left| \nabla_\G d(x) \right|^2 u}{d(x)^2}   \right)  < \infty.$$
			Thus, $$\int_{t_1}^{t_2} \int\limits_{B_r(0)} d(x)^{-2 -\frac{Q-2}{2}}  \left| \nabla_\G d(x) \right|^2 u  < \infty.$$
			Using the fact from \Cref{unboundedness} that  $$u \geq c(r) d(x)^{-\frac{Q-2}{2}}$$ we obtain
			\begin{eqnarray*}
				\int_{t_1}^{t_2} \int\limits_{B_r(0)} d(x)^{-2 -\frac{Q-2}{2}}  \left| \nabla_\G d(x) \right|^2 u 
				& \geq &  c(r) \int_{t_1}^{t_2} \int\limits_{B_r(0)} d(x)^{-Q}  \left| \nabla_\G d(x) \right|^2,
			\end{eqnarray*}
			which is infinite by \Cref{infiniteintegral}.
			This contradicts the existence of $u$.
		\end{enumerate}
	\end{remark}

	\section{Local and global nonexistence}\label{LocalExistence}
	
Now we prove that a local solution for (\ref{MainProblem}) does not exists if $1 + \frac{2}{\alpha} \leq  p$. It should be noted that for $1 + \frac{2}{\alpha} <  p$ we actually prove the nonexistence of supersolutions.

\subsection{Nonexistence for $1+ \frac{2}{\alpha} < p$.}
	Let us suppose on the contrary that there exists a local solution $u$ to (\ref{MainProblem}).
	In view of \Cref{unboundedness} and after a suitable scaling, choose $1> r > 0$ such that $u > 1$ in $[0,T] \times B_0(r)$.
For a compactly supported smooth function $\phi$ on $B_0(r)$, we intend to use $ \frac{|\phi|^2}{u}$ as our test function. The required regularity for $u$ follows from a standard approximation argument that on $[0,T] \times B_0(r)$ the solution $u$ is nothing but the increasing limit of the solutions $u_n$, with $n > r$, to the iterated system \begin{equation*}
			\begin{cases} 
				{u_n}_t(t,x) - \Delta_\G u_n(t,x) - \frac{\lambda \left| \nabla_\G d(x) \right|^2 u_{n-1}}{d(x)^2 + (1/n)} =  u_{n-1}(t,x)^p + f(x), & (t,x) \in (0,T) \times B_0(r), \\
				u_n(0,x) = u(0,x), & x \in B_0(r), \\
				u_n(t, x) = 0, & (t ,x) \in (0,T) \times \partial{B_0(r)}.
			\end{cases}
	\end{equation*} 
Although the following computations are fairly standard, we include the details for clarity. 
Using $ \frac{|\phi|^2}{u}$ as our test function, we obtain 
	$$\int_0^T\int_{B_0(r)} \frac{|\phi|^2}{u}\left(u_t(t,x) - \Delta_\G u(t,x) - \frac{\lambda \left| \nabla_\G d(x) \right|^2 u}{d(x)^2}\right) =  \int_0^T\int_{B_0(r)} u(t,x)^p\frac{|\phi|^2}{u} + f( x) \frac{|\phi|^2}{u}.$$
	Thus,
	$$\int_0^T\int_{B_0(r)} \frac{|\phi|^2}{u}\left(u_t(t,x) - \Delta_\G u(t,x)\right) \geq  \int_0^T\int_{B_0(r)} u(t,x)^p\frac{|\phi|^2}{u}.$$
	Since $\int_{B_0(r)} \log(u(T,x))^{s} < \infty$ for every $s \in [1, \infty)$, an application of Picone's inequality \cite[Lemma 3.1]{RuzhanskySuraganPotentialAnalysis}  gives,
	$$ \int_{B_0(r)} |\phi|^2 \log(u(T,x)) dx + \int_0^T\int_{B_0(r)} |\nabla_\G \phi|^2  dx dt \geq   \int_0^T\int_{B_0(r)} u(t,x)^{p-1}\phi^2 dx dt.$$
	Using \Cref{unboundedness} again, we obtain 
	\begin{equation}\label{Inequality 3.2}
		\int_{B_0(r)} |\phi|^2 \log(u(T,x)) dx + T\int_{B_0(r)} |\nabla_\G \phi|^2  dx \geq  C \int_0^T\int_{B_0(r)} d(x)^{-\alpha(p-1)}\phi^2 dx dt.
	\end{equation}
	From H\"{o}lder's inequality we have
	$$\int_{B_0(r)} |\phi|^2 \log(u(T,x)) dx \leq  \left( \int_{B_0(r)} \log(u(T,x))^{\frac{Q}{2}}  \right)^{\frac{2}{Q}} \left(\int_{B_0(r)} |\phi|^{\frac{2Q}{Q-2}}  \right)^{\frac{Q-2}{Q}} .$$
	Sobolev-type inequality \cite{Folland} now tells us that 
	$$\int_{B_0(r)} |\phi|^2 \log(u(T,x)) dx \leq  \left( \int_{B_0(r)} \log(u(T,x))^{\frac{Q}{2}}  \right)^{\frac{2}{Q}} \left(C \int_{B_0(r)} |\nabla_\G \phi|^2  \right) .$$
	From inequality \ref{Inequality 3.2}, we obtain
	$$  \left( T + C  \left( \int_{B_0(r)} \log(u(T,x))^{\frac{Q}{2}}  \right)^{\frac{2}{Q}}  \right)  \int_{B_0(r)} |\nabla_\G \phi|^2  dx dt \geq  C \int_0^T\int_{B_0(r)}\frac{\phi^2}{d(x)^{\alpha(p-1)}}  dx dt.$$
	This inequality contradicts the Hardy inequality since  $p > 1 + \frac{2}{\alpha}$.
	So no such $u$ exists.

	\subsection{Nonexistence for $1+ \frac{2}{\alpha} = p$}
	
	Let us assume that $\lambda < \left(\frac{Q-2}{2}\right)^2$ because the condition $\lambda = \left(\frac{Q-2}{2}\right)^2$ has already been discussed in (1) of \Cref{IntegralOfDpowerMinusAlphauPowerpIsFinite}.
	Suppose that a local solution $u_1$ exists for (\ref{MainProblem}).
	While dealing with the previous case of $1+ \frac{2}{\alpha} < p$, what helped us in reaching a contradiction to the Hardy inequality was the fact that $u_1 \geq c d(x)^{-\alpha}$ in certain domain.
	In the proof that follows, a similar procedure is followed.
	We first show that some dilation of $u_1$ is bigger than the function $w$ which is defined below and involves another crucial function apart from $d(x)^{-\alpha}$.

	For $0 < r <1$, $x \in B_0(r)$ and any $\beta > 0$, define $$w(t,x) = d(x)^{-\alpha} \left(t^2 \left(\log\left( \frac{1}{d(x)} \right)\right)^\beta + 1 \right) \in C([0,T], L^2(B_0(r))) \cap L^2([0,T], S_0^1(B_0(r))).$$
	A suitable choice of  $\beta$  will lead us to contradicting the Hardy inequality.
	Now,
	$$w^p = d(x)^{-(\alpha + 2)} \left(t^2 \left(\log\left( \frac{1}{d(x)} \right)\right)^\beta + 1 \right)^p,$$
	$$w_t = 2t d(x)^{-\alpha}  \left(\log\left( \frac{1}{d(x)} \right)\right)^\beta$$
	and
	\begin{eqnarray*}
		\Delta_\G w &=& \left(t^2 \left(\log\left( \frac{1}{d(x)} \right)\right)^\beta + 1 \right) \Delta_\G  d(x)^{-\alpha} + d(x)^{-\alpha} \Delta_\G  \left(t^2 \left(\log\left( \frac{1}{d(x)} \right)\right)^\beta + 1 \right) \\
		&+& 2 \nabla_\G d(x)^{-\alpha} \cdot \nabla_\G \left(t^2 \left(\log\left( \frac{1}{d(x)} \right)\right)^\beta + 1 \right)\\
		&=& \left(t^2 \left(\log\left( \frac{1}{d(x)} \right)\right)^\beta + 1 \right) \Delta_\G  d(x)^{-\alpha} + t^2 d(x)^{-\alpha} \Delta_\G  \left( \left(\log\left( \frac{1}{d(x)} \right)\right)^\beta \right) \\
		&+& 2 t^2 \nabla_\G d(x)^{-\alpha} \cdot \nabla_\G \left( \left(\log\left( \frac{1}{d(x)} \right)\right)^\beta  \right).
	\end{eqnarray*}
	For any $s \in (0, \infty)$, set $$g_1(s) = \frac{1}{s}, \ g_2(s) = \log(s) \text{ and } g_3(s) = s^\beta.$$
	Then for
	$$g(s) = (g_3 \circ g_2 \circ g_1)(s) = \left(\log\left(\frac{1}{s}\right)\right)^\beta$$
	we have
	$$g'(s) = g_3'((g_2 \circ g_1)(s)) g_2'(g_1(s)) g_1'(s) =   \frac{-\beta}{s} \left(\log\left(\frac{1}{s}\right)\right)^{\beta - 1} ,$$
	and 
	$$g''(s) = \left(\log\left(\frac{1}{s}\right)\right)^{\beta - 1} \frac{\beta}{s^2} + \frac{\beta}{s^2} (\beta - 1) \left(\log\left(\frac{1}{s}\right)\right)^{\beta - 2} .$$
	Hence,
	\begin{eqnarray*}
		&&\Delta_\G  \left( \left(\log\left( \frac{1}{d(x)} \right)\right)^\beta \right) =  |\nabla_\G d(x)|^2 \left( g''(d(x)) + \frac{Q-1}{d(x)} g'(d(x)))    \right)\\
		&=& |\nabla_\G d(x)|^2 \left( \left(\log\left(\frac{1}{d(x)}\right)\right)^{\beta - 1} \frac{\beta}{d(x)^2} + \frac{\beta}{d(x)^2} (\beta - 1) \left(\log\left(\frac{1}{d(x)}\right)\right)^{\beta - 2} \right. \\
		& +& \left. \frac{-\beta}{d(x)} \frac{Q-1}{d(x)} \left(\log\left(\frac{1}{d(x)}\right)\right)^{\beta - 1}    \right)\\
		&=& \frac{\beta|\nabla_\G d(x)|^2 }{d(x)^2}  \left( \left(\log\left(\frac{1}{d(x)}\right)\right)^{\beta - 1}  +  (\beta - 1) \left(\log\left(\frac{1}{d(x)}\right)\right)^{\beta - 2} \right. \\
		& -& \left. (Q-1) \left(\log\left(\frac{1}{d(x)}\right)\right)^{\beta - 1}     \right)\\
		&=& \left(\log\left(\frac{1}{d(x)}\right)\right)^{\beta - 2} \frac{\beta|\nabla_\G d(x)|^2 }{d(x)^2}  \left( \beta - 1  +  (-Q+2) \left(\log\left(\frac{1}{d(x)}\right)\right)     \right).
	\end{eqnarray*}
	Moreover,
	\begin{eqnarray*}
		\nabla_\G \left(\left(\log\left( \frac{1}{d(x)} \right)\right)^\beta \right) &=&  (g_3 \circ g_2 \circ g_1)'(d(x)) \nabla_\G  d(x) \\
		&=& \frac{-\beta}{d(x)} \left(\log\left(\frac{1}{d(x)}\right)\right)^{\beta - 1}  \nabla_\G  d(x).
	\end{eqnarray*}
	This shows that
	\begin{eqnarray*}
		\Delta_\G w &=& \left(t^2 \left(\log\left( \frac{1}{d(x)} \right)\right)^\beta + 1 \right) \Delta_\G  d(x)^{-\alpha} + t^2 d(x)^{-\alpha} \Delta_\G  \left( \left(\log\left( \frac{1}{d(x)} \right)\right)^\beta \right) \\
		&+& 2 t^2 \nabla_\G d(x)^{-\alpha} \cdot \nabla_\G \left( \left(\log\left( \frac{1}{d(x)} \right)\right)^\beta  \right)\\
		&=& \left(t^2 \left(\log\left( \frac{1}{d(x)} \right)\right)^\beta + 1 \right)  \left( \frac{ -\lambda \left| \nabla_\G d(x) \right|^2}{d(x)^{2 + \alpha} } \right)  \\
		&+& t^2 d(x)^{-\alpha} \left(\log\left(\frac{1}{d(x)}\right)\right)^{\beta - 2} \frac{\beta|\nabla_\G d(x)|^2 }{d(x)^2}  \left( \beta - 1  +  (-Q+2) \left(\log\left(\frac{1}{d(x)}\right)\right)     \right) \\
		&+& 2 \frac{-\beta t^2}{d(x)} \left(\log\left(\frac{1}{d(x)}\right)\right)^{\beta - 1}   \nabla_\G d(x)^{-\alpha} \cdot    \nabla_\G d(x) \\
		&=& \left(t^2 \left(\log\left( \frac{1}{d(x)} \right)\right)^\beta + 1 \right) \left( \frac{ -\lambda \left| \nabla_\G d(x) \right|^2}{d(x)^{2 + \alpha} } \right)   \\
		&+& t^2 d(x)^{-\alpha} \left(\log\left(\frac{1}{d(x)}\right)\right)^{\beta - 2} \frac{\beta|\nabla_\G d(x)|^2 }{d(x)^2}  \left( \beta - 1  +  (-Q+2) \left(\log\left(\frac{1}{d(x)}\right)\right)     \right) \\
		&+& 2 \frac{ \alpha\beta t^2}{d(x)} \left(\log\left(\frac{1}{d(x)}\right)\right)^{\beta - 1}     d(x)^{-\alpha - 1}     |\nabla_\G d(x)|^2 \\
		&=&  \left( \frac{  \left| \nabla_\G d(x) \right|^2}{d(x)^{2 + \alpha} } \right) \times \left(  -\lambda \left(t^2 \left(\log\left( \frac{1}{d(x)} \right)\right)^\beta + 1 \right) \right. \\
		&+& \left. t^2  \left(\log\left(\frac{1}{d(x)}\right)\right)^{\beta - 2} \beta  \left( \beta - 1  +  (-Q+2) \left(\log\left(\frac{1}{d(x)}\right)\right)     \right) \right. \\
		&+& \left. 2 \alpha\beta t^2 \left(\log\left(\frac{1}{d(x)}\right)\right)^{\beta - 1}    \right).
	\end{eqnarray*}
	Thus,
	\begin{eqnarray*}
		&&w_t - \Delta_\G w - \frac{\lambda |\nabla_\G d(x)|^2 w}{d(x)^2}\\
		&=& 2t d(x)^{-\alpha}  \left(\log\left( \frac{1}{d(x)} \right)\right)^\beta - \left( \frac{  \left| \nabla_\G d(x) \right|^2}{d(x)^{2 + \alpha} } \right) \times \left(  -\lambda \left(t^2 \left(\log\left( \frac{1}{d(x)} \right)\right)^\beta + 1 \right) \right. \\
		&+& \left. t^2  \left(\log\left(\frac{1}{d(x)}\right)\right)^{\beta - 2} \beta  \left( \beta - 1  +  (-Q+2) \left(\log\left(\frac{1}{d(x)}\right)\right)     \right) \right. \\
		&+& \left. 2 \alpha\beta t^2 \left(\log\left(\frac{1}{d(x)}\right)\right)^{\beta - 1}    \right) - \frac{\lambda |\nabla_\G d(x)|^2 \left(t^2 \left(\log\left( \frac{1}{d(x)} \right)\right)^\beta + 1 \right)}{d(x)^{2+ \alpha}}\\
		&=& \frac{t}{d(x)^{2+\alpha}} \left\{   2 d(x)^2  \left(\log\left( \frac{1}{d(x)} \right)\right)^\beta -  |\nabla_\G d(x)|^2     \left(  -\lambda \left(t \left(\log\left( \frac{1}{d(x)} \right)\right)^\beta + \frac{1}{t} \right) \right. \right. \\
		&+& \left. \left. t  \left(\log\left(\frac{1}{d(x)}\right)\right)^{\beta - 2} \beta  \left( \beta - 1  +  (-Q+2) \left(\log\left(\frac{1}{d(x)}\right)\right)     \right) \right. \right. \\
		&+& \left. \left. 2 \alpha\beta t \left(\log\left(\frac{1}{d(x)}\right)\right)^{\beta - 1}    +  \lambda  \left(t \left(\log\left( \frac{1}{d(x)} \right)\right)^\beta + \frac{1}{t} \right)  \right) \right\}\\
		&=& \frac{t}{d(x)^{2+\alpha}} \left\{   2 d(x)^2  \left(\log\left( \frac{1}{d(x)} \right)\right)^\beta  -   |\nabla_\G d(x)|^2      t \beta \left(\log\left(\frac{1}{d(x)}\right)\right)^{\beta - 1}   \times   \right.  \\
		&\times& \left. \left( (\beta - 1) \left(\log\left(\frac{1}{d(x)}\right)\right)^{-1}   -Q+2      + 2 \alpha  \right) \right\}
	\end{eqnarray*}
	It is known  (see \cite[Proposition 3.2]{SuraganTalwar2}) that $  |\nabla_\G d(x)|^2$ is bounded.
	Moreover, $r$ can be chosen small enough so that the term $\left( (\beta - 1) \left(\log\left(\frac{1}{d(x)}\right)\right)^{-1}   -Q+2      + 2 \alpha  \right)$ is negative for any given $\beta$.
	Thus, there exists $T_1 \in (0, T)$ such that $$w_t - \Delta_\G w - \frac{\lambda |\nabla_\G d(x)|^2 w}{d(x)^2} \leq \beta h(t,x) w^p \text{ in } \ (0,T_1) \times B_0(r),$$
	where $$h(t,x) = \left(  t^2 \log\left( \frac{1}{d(x)} \right)^\beta + 1  \right)^{1-p}.$$
	Hence,
	$$w_t - \Delta_\G w - \frac{\lambda |\nabla_\G d(x)|^2 w}{d(x)^2} \leq \beta h(t,x) w^p + f(x) \text{ in } \ (0,T_1) \times B_0(r).$$
	For $u = c_1 u_1$, with $c_1$ large and $\beta$ small enough, we have $$u_t - \Delta_\G u - \frac{\lambda |\nabla_\G d(x)| u}{d(x)^2} \geq  c_1^{1-p}  u^p + c_1f \geq \beta u^{p} h(t,x) + c_1f.$$
	Using Kato's inequality  for stratified Lie groups \cite{AmbrosioMitidieri}, just as in \cite[Theorem 3.1]{AbdellaouiAlonsoRamos}, it follows that $u \geq w$.
	As in the previous subsection, we obtain $$  T \int_{B_0(r)} |\nabla_\G \phi|^2  dx dt \geq  C \int_0^T\int_{B_0(r)} u^{p-1} \phi^2  dx dt \geq  C \int_0^T\int_{B_0(r)} w^{p-1} \phi^2  dx dt.$$
	Thus,
	$$T \int_{B_0(r)} |\nabla_\G \phi|^2  dx dt \geq  C   \int_0^T\int_{B_0(r)} d(x)^{-2}  \left(t^2 \log\left( \frac{1}{d(x)} \right)^\beta + 1  \right)^{p-1} \phi^2  dx dt.$$
	This shows that
	$$ \int_{B_0(r)} |\nabla_\G \phi|^2  dx dt \geq  C   \int_{B_0(r)} \frac{\phi^2}{d(x)^{2}}  \log\left( \frac{1}{d(x)} \right)^{\beta(p-1)}   dx dt,$$
	which contradicts the Hardy inequality.
	
Now that it is confirmed that a local solution to (\ref{MainProblem}) may exists only if $1 < p < 1+ \frac{2}{\alpha}$, we consider only this case to study the global non-existence.
The following two theorems show that no global solution exists when $p \leq \frac{Q-\alpha}{Q-2-\alpha}$.
Note that this exponent differs from the one obtained in \cite{AbdellaouiPearlPrimo} and the sole  reason is that the forcing term $f$ is independent of the time variable, as opposed to ibid.
It is expected that if $f$ depends on the time variable as well, then non-existence of global solutions will follow for $p \leq \frac{Q + 2 - \alpha}{Q-\alpha}$ and a solution may be provided in the spirit of \cite{AbdellaouiSiclariPrimo, AbdellaouiPearlPrimo}.
	
\begin{theorem}\label{subcritical}
Let  $p < \frac{Q-\alpha}{Q-2-\alpha}$, $u_0 \geq 0$ and $f \neq 0$.
Then a global in time weak solution for (\ref{MainProblem2}) does not exist.
	\end{theorem}
	\begin{proof}
Suppose on the contrary that a global in time weak solution $u$ for (\ref{MainProblem2}) exists.
		Then for every test function $\psi$ and every $T > 0$, using the non-negativity of $\psi$ and $u_0$, we obtain \begin{eqnarray*}\label{Inequality}
			\int_0^T \int_{\G} \left(d^{-\alpha(p+1)} u(t,x)^p \right. &+& \left. d^{-\alpha}f(x)\right) \psi \\
			&\leq& - \int_0^T \int_{\G} d^{-2\alpha} u \psi_t - \int_0^T \int_{\G} u \left(\nabla_\G (d^{-2\alpha} \nabla_\G  \psi)\right).
		\end{eqnarray*}
		For an appropriate test function $\psi$, to be chosen at a later stage, applying $\frac{p}{2}-$Young inequality gives
		\begin{eqnarray*}
			\int_{0}^T \int_{\G} d^{-2\alpha} u|\psi_t|  &\leq&  \int_{0}^T \int_{\G} d^{-\alpha\left( \frac{p+1}{p}\right)}u \psi^{\frac{1}{p}} \psi^{\frac{-1}{p}}  |\psi_t| d^{\alpha\left( \frac{1-p}{p}\right)} \\
			& \leq & \frac{1}{2} \int_{0}^T \int_{\G} d^{-\alpha(p+1)}u^p \psi + \frac{1}{p' (\frac{p}{2})^{p'-1}} \int_{0}^T \int_{\G} d^{-\alpha} \psi^{\frac{-1}{p-1}}|\psi_t|^{\frac{p}{p-1}}
		\end{eqnarray*}
		and 
		\begin{eqnarray*}
			\int_{0}^T \int_{\G} u \left(\nabla_\G (d^{-2\alpha} \nabla_\G  \psi)\right) &=& \int_{0}^T \int_{\G} d^{-\alpha\left( \frac{p+1}{p}\right)} u \psi^{\frac{1}{p}} d^{\alpha\left( \frac{p+1}{p}\right)} \psi^{\frac{-1}{p}}   \left(\nabla_\G (d^{-2\alpha} \nabla_\G  \psi)\right)  \\
			& \leq & \frac{1}{2} \int_{0}^T \int_{\G} u^p d^{-\alpha\left(p+1 \right)} \psi \\
			& + &  \frac{1}{p' (\frac{p}{2})^{p'-1}} \int_{0}^T \int_{\G}\psi^{\frac{-1}{p-1}}  d^{\alpha\left( \frac{p+1}{p-1}\right)}   \left(\nabla_\G (d^{-2\alpha} \nabla_\G  \psi)\right) ^{\frac{p}{p-1}}.
		\end{eqnarray*}
		Thus,
		\[\int_{0}^T \int_{\G} d^{-\alpha} f \psi \leq \frac{1}{p' (\frac{p}{2})^{p'-1}} \Big( \int_{0}^T \int_{\G} |\psi|^{\frac{-1}{p-1}} d^{\alpha\left( \frac{p+1}{p-1}\right)} \left(\nabla_\G (d^{-2\alpha} \nabla_\G  \psi)\right) ^{\frac{p}{p-1}} + d^{-\alpha} |\psi|^{\frac{-1}{p-1}}|\psi_t|^{\frac{p}{p-1}} \Big).\]

		In order to choose the a test function $\psi$ which will help in reaching a contradiction, set $$P(x) = \frac{d(x)}{\sqrt{T}}.$$
		Then $$P(x) = (p \circ d)(x)$$ where  $p: \R \to \R$ is given by $$p(w)= \frac{w}{\sqrt{T}}.$$
		Define a test function with separated variables as $$\psi(t,x) = \Phi(P(x)) \Phi(2t/T),$$ where \begin{equation*}
			\Phi(z) = \begin{cases} 
				1, & z \in [0,1],\\
				\searrow, & z \in [1, 2],\\
				0, & z \in (2, \infty),
			\end{cases}
		\end{equation*}
		is a function in $C^2_0(\R^+)$ satisfying the following weighted estimates for every large $T>0$, $$ \int\limits_{B_2'(0)} d^{-2 \alpha}(x') \left|\frac{\left(\Phi' + \Phi''\right)^p(d(\delta_{\sqrt{T}} (x')))}{\Phi(d(\delta_{\sqrt{T}} (x')))}\right|^{\frac{1}{p-1}} dx' < \infty $$  and $$\int_{0}^2  |\Phi(t)|^{\frac{-1}{p-1}}|  \Phi'(t)|^{\frac{p}{p-1}} \int\limits_{B_2'(0)}  d^{-\alpha}(x) |\Phi( d(x))|  < \infty.$$
		Note that $$\text{supp}(\psi) \seq \left\{ (t,x) \in \R^+ \times \G : 0 \leq \frac{2t}{T} \leq 2, 0 \leq P(x) \leq 2 \right\}.$$
		Now $\psi_t(x,t) = \frac{\partial}{\partial t}(\Phi(P(x)) \Phi(\frac{2t}{T})) = \frac{2}{T}  \Phi(P(x)) \Phi'(\frac{2t}{T})$ implies that $$\text{supp}(\psi_t) \seq \left\{ (t,x) \in \R^+ \times \G : 1 \leq \frac{2t}{T} \leq 2, 0 \leq P(x) \leq 2 \right\}.$$
		We now find bounds on the two terms which are important to us.
		With variables $t' = \frac{2t}{T}$ and $x' = \delta_{\frac{1}{\sqrt{T}}} x$, we have
		\begin{eqnarray*}
			\int_{0}^T \int_{\G} d^{-\alpha}(x)  |\psi|^{\frac{-1}{p-1}}|\psi_t|^{\frac{p}{p-1}} &=&  \int_{0}^T \int_{\G} d^{-\alpha}(x) \left|\Phi(P(x)) \Phi\left(\frac{2t}{T}\right)\right|^{\frac{-1}{p-1}} \left|\frac{2}{T}  \Phi(P(x)) \Phi'\left(\frac{2t}{T}\right)\right|^{\frac{p}{p-1}} \\
			& = &  \left(\frac{2}{T}\right)^{{\frac{p}{p-1}}} \int_{0}^T \int_{\G} d^{-\alpha}(x)  |\Phi(P(x))|  \left|\Phi\left(\frac{2t}{T}\right)\right|^{\frac{-1}{p-1}}\left|  \Phi'\left(\frac{2t}{T}\right)\right|^{\frac{p}{p-1}}\\
			& = & \left(\frac{2}{T}\right)^{{\frac{p}{p-1}}} \frac{T^{1+\frac{Q}{2} }}{2} \int\limits_{0}^2 |\Phi(t')|^{\frac{-1}{p-1}}|  \Phi'(t')|^{\frac{p}{p-1}} \times \\
			& & \ \ \ \ \ \ \ \  \ \ \ \ \ \ \ \ \ \ \ \ \ \ \ \ \ \ \ \ \times  \int\limits_{B_2'(0)}  d^{-\alpha}(\delta_{\sqrt{T}}x') |\Phi(d(x'))|  \\
			&=& T^{\frac{Q+2 - \alpha}{2} - \frac{p}{p-1}}  \int_{0}^2  |\Phi(t')|^{\frac{-1}{p-1}}|  \Phi'(t')|^{\frac{p}{p-1}} \int\limits_{B_2'(0)}  d^{-\alpha}(x') |\Phi (d(x'))| \\
			&=& C T^{\frac{Q+2 - \alpha}{2} - \frac{p}{p-1}}.
		\end{eqnarray*}
		Also,
		\begin{eqnarray*}
			\nabla_\G \psi &=& \Phi(2t/T) \nabla_\G \left(\Phi(p \circ d)(x) \right) \\
			&=& \Phi(2t/T) (\Phi \circ p)'(d(x)) \nabla_\G \left(d(x) \right)\\
			&=& \Phi(2t/T) \frac{1}{\sqrt{T}} \Phi'(P(x)) \nabla_\G \left(d(x) \right),
		\end{eqnarray*}
		and hence
		\begin{eqnarray*}
			&&\int_{0}^T \int_{\G} d(x)^{\alpha\left( \frac{p+1}{p-1}\right)}  |\psi|^{\frac{-1}{p-1}}\left(\nabla_\G (d(x)^{-2\alpha} \nabla_\G  \psi)\right)^{\frac{p}{p-1}} \\
			&=&  \frac{1}{\sqrt{T}^{\frac{p}{p-1}}} \int_{0}^T \int_{\G}  \Phi\left(\frac{2t}{T}\right)  d(x)^{\alpha\left( \frac{p+1}{p-1}\right)} \left|\Phi(P(x))\right|^{\frac{-1}{p-1}}  \left| \nabla_\G \left( d^{-2 \alpha}(x) \Phi'(P(x))  \nabla_\G d(x) \right)        \right|^{\frac{p}{p-1}}.
		\end{eqnarray*}
		Note first that
		\begin{eqnarray}
			& & \nabla_\G \left( d^{-2 \alpha} \Phi'(P(x))  \nabla_\G d(x) \right)  \label{Use} \\
			&=& \sum_{i=1}^m \left( d^{-2 \alpha}(x) \Phi'(P(x)) X_i^2(d(x)) + d^{-2 \alpha}(x) \frac{\Phi''(P(x))}{\sqrt{T}} |X_i(d(x))|^2  \right. \nonumber \\ 
			& + & \left.  X_i(d^{-2 \alpha})(x) \Phi'(P(x)) X_i(d(x)) \right). \nonumber
		\end{eqnarray}
		Also,  there exists $C > 0$ such that for every $x \in \G \setminus \{0\}$ we have
		\begin{eqnarray*}
			d^{-2 \alpha}(\delta_{\sqrt{T}}(x))  &=&  T^{- \alpha} d^{-2 \alpha}(x), \ \ \Phi'(P(\delta_{\sqrt{T}}(x))) \leq  C,\\
			X_i^2(d(\delta_{\sqrt{T}}(x))) & = &  \frac{X_i^2 d(x)}{\sqrt{T}}, \ \ \frac{\Phi''(P(\delta_{\sqrt{T}}(x)))}{\sqrt{T}}  \leq  \frac{C}{\sqrt{T}},\\
			|X_i(d(\delta_{\sqrt{T}}(x))|^2 & \leq & | X_i(d(x)) |^2,\\
			X_i(d^{-2 \alpha})(\delta_{\sqrt{T}}(x)) & = & T^{-1/2} X_i (d^{-2 \alpha}(\delta_{\sqrt{T}}(x))) = T^{- \alpha - 1/2}X_i (d^{-2 \alpha}(x)) \  \text{and} \\
			\Phi'(P(\delta_{\sqrt{T}}(x)))  & \leq & C.
		\end{eqnarray*}
		Hence,
		\begin{eqnarray*}
			& &\left|\nabla_\G \left( d^{-2 \alpha}(\delta_{\sqrt{T}} (x)) \Phi'(P(\delta_{\sqrt{T}}(x)))  \nabla_\G d(\delta_{\sqrt{T}}(x)) \right) \right|  \\
			&\leq& C T^{- \alpha - \frac{1}{2}} \sum_{i=1}^m \left( d^{-2 \alpha}(x)  |X_i^2(d(x))| +  d^{-2 \alpha}(x) |X_i(d(x)|^2   +  |X_i(d^{-2 \alpha}(x))|  |X_i(d(x))| \right). 
		\end{eqnarray*}
		Using $x \to \delta_{\sqrt{T}} (x'), t \to \frac{2t'}{T}$ and Equation \ref{Use}, we obtain
		\begin{eqnarray*}
			&&\int_{0}^T \int_{\G} d(x)^{\alpha\left( \frac{p+1}{p-1}\right)} |\psi|^{\frac{-1}{p-1}}\left(\nabla_\G (d(x)^{-2\alpha} \nabla_\G  \psi)\right) ^{\frac{p}{p-1}} \\
			&\leq & C \frac{T^{\frac{Q}{2}+1}}{\sqrt{T}^{\frac{p}{p-1}}} \int_{0}^2 \int\limits_{B_2'(0)}  d(\delta_{\sqrt{T}} (x'))^{\alpha\left( \frac{p+1}{p-1}\right)} \left|\Phi(d(\delta_{\sqrt{T}} (x')))\right|^{\frac{-1}{p-1}} \Phi\left(t'\right) \times \\
			& & \ \ \ \ \ \ \ \ \ \ \ \ \ \ \ \ \ \ \ \ \ \ \ \ \times  \left| \nabla_\G \left( d^{-2 \alpha}(\delta_{\sqrt{T}} (x')) \left(\Phi\right)(P(\delta_{\sqrt{T}}(x')))  \nabla_\G d(\delta_{\sqrt{T}} (x')) \right)        \right|^{\frac{p}{p-1}}\\
			&\leq&  C T^{\frac{Q}{2} + 1 +  \frac{\alpha(p+1)}{2(p-1)}   - \frac{p}{2(p-1)}  - \left(\alpha + \frac{1}{2}\right)\left( \frac{p}{p-1} \right) }  \int\limits_{B_2'(0)} d^{-2 \alpha}(x') \left|\frac{\left(\Phi' + \Phi''\right)^p(d(\delta_{\sqrt{T}} (x')))}{\Phi(d(\delta_{\sqrt{T}} (x')))}\right|^{\frac{1}{p-1}} dx'.
		\end{eqnarray*}
		As, \begin{eqnarray*}
			\int_{0}^T \int_{\G} d^{-\alpha}(x) f \psi &=& \int_{0}^T \Phi(2t/T) \int_{\G} d^{-\alpha}(x) f \Phi(P(x))\\
			& = &  \frac{T}{2} \int_{0}^2 \Phi(t') \int_{\G} d^{-\alpha}(x) f \Phi(P(x)) \geq T \int_{\G} d^{-\alpha}(x) f \Phi(P(x))
		\end{eqnarray*}
		we obtain $$T \int_{\G} d^{-\alpha}(x) f \Phi(P(x)) \leq \int_{0}^T \int_{\G} d^{-\alpha}(x) f \psi \leq C_1 (T^{\frac{Q+2-\alpha}{2} - \frac{p}{p-1}} +  T^{\frac{Q}{2} + 1 -  \frac{\alpha}{2}   - \frac{p}{p-1}  }).$$
		Equivalently, $$\int_{\G} d^{-\alpha}(x) f \Phi(P(x)) \leq T^{-1} \int_{0}^T \int_{\G} d^{-\alpha}(x) f \psi \leq C_1 (T^{\frac{Q-\alpha}{2} - \frac{p}{p-1}} ).$$
		Because $u$ is assumed to be a global solution, one may take $T \to \infty$.
		Then $ \int_{\G} d^{-\alpha}(x) f(x) \Phi \to \int_{\G} d^{-\alpha}(x) f(x) $ by dominated convergence theorem while the right hand side goes to zero since $p < \frac{Q - \alpha}{Q-2 - \alpha}$ implies $$\frac{Q-\alpha}{2} - \frac{p}{p-1} < 0.$$ 
This shows that $\int_{\G} d^{-\alpha}(x) f(x) \leq 0$, contradicting the  hypothesis that $f$ is non-zero and non-negative.
So, no such $u$ exists and the statement of the theorem is proved.
\end{proof}

\begin{theorem}\label{Critical}
	Let $p = \frac{Q - \alpha}{Q-2 - \alpha}$, $u_0 \geq 0$ and $0< \int_{\G} d^{-\alpha}(x) f(x) dx$.
	Then (\ref{MainProblem2}) does not admit a global in time weak solution.
\end{theorem}
\begin{proof}
	We prove this statement by the method of contradiction.
	Let us suppose that $u$ is a global in time weak solution for (\ref{MainProblem2}).
	As in \Cref{subcritical}, for a suitable test function $\psi$, there exists a constant $C$ such that 
	\begin{equation}\label{CriticalInequalityOne}
		\int_{0}^T \int_{\G}  d^{-\alpha} f \psi \leq C \Big( \int_{0}^T \int_{\G}  |\psi|^{\frac{-1}{p-1}} \left(  d^{\alpha\left( \frac{p+1}{p-1}\right)} \left(\nabla_\G (d^{-2\alpha} \nabla_\G  \psi)\right) ^{\frac{p}{p-1}} + d^{-\alpha} |\psi_t|^{\frac{p}{p-1}} \right)  \Big).
	\end{equation}
	For any $1 < R < \infty$, set $$P(x) := \frac{\ln\left(\frac{d(x)}{\sqrt{R}}\right)}{\ln{\sqrt{R}}}$$ and define a test function with separated variables as $$\psi(t,x) = \Phi(P(x)) \Phi\left(t/T \right)$$ where \begin{equation*}
		\Phi(z) = \begin{cases} 
			1, & z \in (- \infty ,0],\\
			\searrow, & z \in [0, 1],\\
			0, & z \in (1, \infty),
		\end{cases}
	\end{equation*}
	is a function in $C^\infty(\R)$ satisfying the estimates $$\int_{0}^1 \left| \frac{\Phi'^p\left(t\right)}{\Phi\left(t\right)}\right|^{\frac{1}{p-1}} dt < \infty$$ and $$\int\limits_{1 \leq d(x) \leq \sqrt{R}}  \left|\frac{(\Phi')^p\left( \frac{\ln(d(x))}{\ln(\sqrt{R})} \right) + (\Phi'')^p\left( \frac{\ln(d(x))}{\ln(\sqrt{R})} \right)}{\Phi\left( \frac{\ln(d(x))}{\ln(\sqrt{R})} \right)} \right|^{\frac{1}{p-1}} dx < \infty$$ where the integrals are independent of $R$.

Note that $$\text{supp}(\psi) \seq \{ (t,x) \in \R^+ \times \G : 0 \leq t \leq T, - \infty < P(x) \leq 1 \}.$$
For $\sqrt{R} \leq d(x)$ we have,
	\begin{eqnarray*}
		- \infty < P(x) \leq 1 &\iff&   - \infty <  \ln\left(\frac{d(x)}{\sqrt{R}}\right) \leq \ln{\sqrt{R}}\\
		& \iff&  0 < \exp{\left(\ln\left(\frac{d(x)}{\sqrt{R}}\right)\right)} \leq \sqrt{R}     \\
		&   \iff &  0 < d(x) \leq R
	\end{eqnarray*}
while for $d(x) \leq \sqrt{R}$ we have,
\begin{eqnarray*}
	- \infty < P(x) \leq 1 &\iff&  - 1 \leq \frac{\ln\left(\frac{\sqrt{R}}{d(x)}\right)}{\ln{\sqrt{R}}} < \infty\\
	& \iff& - \ln{\sqrt{R}} \leq \ln\left(\frac{\sqrt{R}}{d(x)}\right) < \infty\\
	& \iff&  \frac{1}{\sqrt{R}} \leq \frac{\sqrt{R}}{d(x)} < \infty    \\
	&   \iff &  0 < d(x) \leq R
\end{eqnarray*}
which confirms that the set $\text{supp}(\psi)$ is compact.
The fact that $$\psi_t(t,x) = \frac{\partial}{\partial t}\left(\Phi(P(x)) \Phi\left(\frac{t}{T}\right)\right) = \frac{\Phi(P(x))}{T}   \Phi'\left(\frac{t}{T}\right)$$ implies the containment $$\text{supp}(\psi_t) \seq \{ (t,x) \in \R^+ \times \G : 0 \leq t \leq T, 0 < d(x) \leq R \}.$$
	
Using the estimates on $\psi$, for $(t,x) \in \R^+ \times \G$, $x'= \delta_{\frac{1}{\sqrt{R}}}(x)$ and $t' = \frac{t}{T}$, we obtain
\begin{eqnarray*}
&&\int_{0}^T \int\limits_{\text{supp}(\psi_t)}  d^{-\alpha}(x) \psi(t,x)^{\frac{-1}{p-1}}\psi_t(t,x)^{\frac{p}{p-1}} dx dt \\
&=&  \int_{0}^T \int\limits_{\text{supp}(\psi_t)} d^{-\alpha}(x) \left|\Phi(P(x)) \Phi\left(\frac{t}{T}\right)\right|^{\frac{-1}{p-1}} \left|\frac{\Phi(P(x)) \Phi'(t/T)}{T}   \right|^{\frac{p}{p-1}} dx dt \\
& = & R^{\frac{Q}{2}} T \left(\frac{1}{T}\right)^{\frac{p}{p-1}} \int_{0}^1 \int\limits_{0< d(x') < \sqrt{R}} d^{-\alpha}(\delta_{\sqrt{R}}(x')) \Phi(P(\delta_{\sqrt{R}}(x')))  \left| \frac{\Phi'^p\left(t'\right)}{\Phi\left(t'\right)}\right|^{\frac{1}{p-1}} dx' dt'\\
& = & T^{\frac{-p}{p-1}} T R^{\frac{Q-\alpha}{2}} \int_{0}^1 \int\limits_{0< d(x') < \sqrt{R}} d^{-\alpha}(x') \Phi\left( \frac{\ln(d(x'))}{\ln(\sqrt{R})} \right)   \left| \frac{\Phi'^p\left(t'\right)}{\Phi\left(t'\right)}\right|^{\frac{1}{p-1}} dx' dt'\\
& \leq & T^{\frac{-p}{p-1}} T R^{\frac{Q-\alpha}{2}} R^{\frac{-\alpha}{2}} \int_{0}^1   \left| \frac{\Phi'^p\left(t'\right)}{\Phi\left(t'\right)}\right|^{\frac{1}{p-1}} dt' \int\limits_{0< d(x') < \sqrt{R}} \Phi\left( \frac{\ln(d(x'))}{\ln(\sqrt{R})} \right) dx'  \\
& \leq & C T^{\frac{-p}{p-1}} T R^{\frac{Q-\alpha}{2}} R^{\frac{-\alpha}{2}}.
	\end{eqnarray*}
	From  inequality \ref{CriticalInequalityOne} and the fact that \begin{eqnarray*}
		\int_{0}^T \int_{\G} d^{-\alpha}(x) f(x) \psi(t,x) dx dt &=& \int_{0}^T \Phi\left(\frac{t}{T} \right) \int_{\G} d^{-\alpha}(x) f(x) \Phi(P(x)) dx dt \\
		&  = & T \int_{0}^1 \Phi(t') \int_{\G} d^{-\alpha}(x) f(x) \Phi(P(x)) dx dt' \\
		&=&  \geq  C_2 T \int_{\G} d^{-\alpha}(x) f(x) \Phi(P(x)) dx
	\end{eqnarray*}
	we obtain \begin{equation}\label{CriticatlInequality}
		\int_{\G} d^{-\alpha}(x) f(x) \Phi(P(x)) \leq C' \left(\int_{0}^T \int_{\G} \frac{1}{T}  |\psi|^{\frac{-1}{p-1}}      d^{\alpha\left( \frac{p+1}{p-1}\right)} \left(\nabla_\G (d^{-2\alpha} \nabla_\G  \psi)\right) ^{\frac{p}{p-1}} +  T^{\frac{-p}{p-1}} R^{\frac{Q-\alpha}{2}} R^{\frac{-\alpha}{2}} \right).
	\end{equation}
Since $\frac{-p}{p-1} + \frac{Q-\alpha}{2} = 0$ substituting $T = R$ and taking $R \to \infty$ shows that $$T^{\frac{-p}{p-1}} R^{\frac{Q-\alpha}{2}} R^{\frac{-\alpha}{2}} \to 0.$$
It thus suffices to prove that $$\int_{0}^T \int_{\G} \frac{1}{T}  |\psi(x)|^{\frac{-1}{p-1}}   d^{\alpha\left( \frac{p+1}{p-1}\right)}(x) \left(\nabla_\G (d^{-2\alpha}(x) \nabla_\G  \psi(x))\right) ^{\frac{p}{p-1}} \to 0$$ as $T = R \to \infty$.

For this, let us first find $\text{supp}\left(\nabla_\G \left(d(x)^{-2\alpha} \nabla_\G  \psi(t,x) \right)\right)$.
With $$p(s) = \frac{s}{\sqrt{R}} \text{ and } \Phi_1(s) = \frac{s}{\ln \sqrt{R}}$$ we have,
$$P(x) = \left(\Phi_1 \circ \ln \circ p \right)(d(x)).$$
Then,
\begin{eqnarray*}
&&\nabla_\G \left(d(x)^{-2\alpha} \nabla_\G  \psi(t,x) \right)
		= \Phi\left(\frac{t}{T}\right) \nabla_\G \left(  d(x)^{-2\alpha} \nabla_\G  \left(\Phi \circ \Phi_1 \circ \ln \circ p \right)(d(x)) \right) \\
		&=&  \frac{1}{\ln{\sqrt{R}}} \Phi\left(\frac{t}{T}\right) \nabla_\G \left(d(x)^{-2\alpha}  \left( \Phi'(\Phi_1 \circ \ln \circ p)(d(x)) \frac{1}{p(d(x))} \frac{1}{\sqrt{R}} \right)         \nabla_\G  d(x)  \right)\\
		&=&  \frac{1}{\ln{\sqrt{R}}} \Phi\left(\frac{t}{T}\right) \nabla_\G \left(d(x)^{-2\alpha - 1}  \Phi'(P(x))  \nabla_\G  d(x)   \right).
	\end{eqnarray*}
	Hence, $\nabla_\G \left(d(x)^{-2\alpha} \nabla_\G  \psi(t,x) \right) \neq 0 $ only if \begin{eqnarray*}
	\Phi'(P(x)) \neq 0 \iff	0 < P(x) \leq 1  \iff \sqrt{R} < d(x) \leq R.
	\end{eqnarray*}
	Thus,
	$$\text{supp}\left(\nabla_\G \left(d(x)^{-2\alpha} \nabla_\G  \psi(t,x) \right)\right) \seq K:= \{ (t,x) \in \R^+ \times \G : 0 \leq t \leq T, \sqrt{R} < d(x) \leq  R \}.$$

	Moreover, since \begin{eqnarray*}
&& \nabla_\G \left(d(x)^{-2\alpha - 1}  \Phi'(P(x))  \nabla_\G  d(x)   \right) = \sum_{i=1}^m X_i \left(d(x)^{-2\alpha - 1}  \Phi'(P(x)) X_i  d(x)   \right) \\
&=& \sum_{i=1}^m \left(   d(x)^{-2\alpha - 1}  \Phi'(P(x))  X_i^2  d(x) \right. \\ 
&+ &  \left.   \left( d(x)^{-2\alpha - 1} X_i\left( \left(\Phi' \circ \Phi_1 \circ \ln \circ p \right)(d(x))   \right) (d(x)) + X_i (d(x)^{-2\alpha - 1}) \Phi'(P(x))  \right) X_i(d(x))   \right) \\
&=& \sum_{i=1}^m  d(x)^{-2\alpha - 1}  \Phi'(P(x))  X_i^2  d(x)  \\ 
&+ &    \frac{1}{\ln{\sqrt{R}}} d(x)^{-2\alpha - 1} \frac{\Phi''(P(x))}{d(x)} | X_i( d(x) )|^2 + X_i (d(x)^{-2\alpha - 1}) \Phi'(P(x)) X_i(d(x)),
\end{eqnarray*}
there exists $C > 0$ such that for every $x \in \G$ we have
\begin{eqnarray*}
d^{-2 \alpha - 1}(\delta_{\sqrt{R}}(x))
&=&  R^{- \alpha - \frac{1}{2}} d^{-2 \alpha - 1}(x), \ \
\Phi'(P(\delta_{\sqrt{R}}(x)))  \leq  C,\\
X_i^2(d(\delta_{\sqrt{R}}(x))) & = &  \frac{X_i^2 d(x)}{\sqrt{R}}, \ 
\frac{\Phi''(P(\delta_{\sqrt{R}}(x)))}{d(\delta_{\sqrt{R}}(x))}  \leq  \frac{C}{\sqrt{R} d(x)}, \\
|X_i(d(\delta_{\sqrt{R}}(x))|^2  & \leq &   | X_i(d(x)) |^2, \ \ \Phi'(P(\delta_{\sqrt{R}}(x)))   \leq  C \  \text{and} \\
X_i(d^{-2 \alpha - 1})(\delta_{\sqrt{R}}(x)) & = & R^{-1/2} X_i (d^{-2 \alpha - 1 }(\delta_{\sqrt{R}}(x))) = R^{- \alpha - 1}X_i (d^{-2 \alpha - 1}(x)).
\end{eqnarray*}
This implies that 
\begin{eqnarray*}
		&& \nabla_\G \left(d(\delta_{\sqrt{R}}(x))^{-2\alpha - 1}  \Phi'(P(\delta_{\sqrt{R}}(x)))  \nabla_\G  d(\delta_{\sqrt{R}}(x))   \right) \\
		&\leq  & C       R^{- \alpha - 1}   \nabla_\G \left(d(x)^{-2\alpha - 1}  \Phi'(P(\delta_{\sqrt{R}}(x)))  \nabla_\G  d(x)   \right)
	\end{eqnarray*}

	Using $p = \frac{Q - \alpha}{Q-2 - \alpha}$, $t' = \frac{t}{T}$ and $x = \delta_{\sqrt{R}}(x')$, we obtain
	\begin{eqnarray*}
		&& \int_{0}^T \int\limits_{K} \frac{1}{T} |\psi(t,x)|^{\frac{-1}{p-1}} d(x)^{\alpha\left( \frac{p+1}{p-1}\right)} \left(\nabla_\G (d(x)^{-2\alpha} \nabla_\G  \psi(t,x)\right) ^{\frac{p}{p-1}} dx dt\\
		& = &  \left( \frac{1}{\ln{\sqrt{R}}}\right)^{\frac{p}{p-1}} \int_{0}^T \int\limits_{K} \frac{1}{T} |\psi(t,x)|^{\frac{-1}{p-1}} d(x)^{\alpha\left( \frac{p+1}{p-1}\right)} \times \\
		&& \ \ \  \ \ \ \ \ \ \  \  \ \ \ \ \ \ \  \  \ \ \ \ \ \ \  \  \ \times  \left( \Phi\left(\frac{t}{T}\right) \nabla_\G \left(d(x)^{-2\alpha - 1}  \Phi'(P(x))  \nabla_\G  d(x)   \right) \right) ^{\frac{p}{p-1}} dx dt\\
& = &  \left( \frac{1}{\ln{\sqrt{R}}}\right)^{\frac{p}{p-1}} \int_{0}^T \int\limits_{K} \frac{1}{T} \Phi\left(\frac{t}{T}\right) |\Phi(P(x))|^{\frac{-1}{p-1}} d(x)^{\alpha\left( \frac{p+1}{p-1}\right)} \times \\
		&& \ \ \  \ \ \ \ \ \ \  \  \ \ \ \ \ \ \  \  \ \ \ \ \ \ \  \  \ \times  \left( \nabla_\G \left(d(x)^{-2\alpha - 1}  \Phi'(P(x))  \nabla_\G  d(x)   \right) \right) ^{\frac{p}{p-1}} dx dt\\
		& = &  \left( \frac{1}{\ln{\sqrt{R}}}\right)^{\frac{p}{p-1}} R^{\frac{Q}{2}} \int_{0}^1 \Phi\left(t'\right) dt' \int\limits_{B_{\sqrt{R}}'(0) \setminus B_1'(0)} d(\delta_{\sqrt{R}}(x'))^{\alpha\left( \frac{p+1}{p-1}\right)}     \left(\Phi(P(\delta_{\sqrt{R}}(x'))) \right)^{\frac{-1}{p-1}}      \times  \\
		&& \ \ \ \ \ \ \ \ \ \ \ \  \times \left( \nabla_\G \left(d(\delta_{\sqrt{R}}(x'))^{-2\alpha - 1}  \Phi'(P(\delta_{\sqrt{R}}(x')))  \nabla_\G  d(\delta_{\sqrt{R}}(x'))   \right) \right) ^{\frac{p}{p-1}} dx'  \\
		& = & C \left( \frac{1}{\ln{\sqrt{R}}}\right)^{\frac{p}{p-1}} R^{\frac{Q}{2} + \left( \frac{\alpha(p+1)}{2(p-1)}\right)  } \int\limits_{B_{\sqrt{R}}'(0) \setminus B_1'(0)} d(x')^{\alpha\left( \frac{p+1}{p-1}\right)} \left|\Phi\left( \frac{\ln(d(x'))}{\ln(\sqrt{R})} \right) \right|^{\frac{-1}{p-1}} \times \\
		&&   \ \ \ \ \ \ \ \ \ \ \  \ \ \ \ \ \ \times   \left(  \nabla_\G \left(d(\delta_{\sqrt{R}}(x'))^{-2\alpha - 1}  \Phi'\left( \frac{\ln(d(x'))}{\ln(\sqrt{R})} \right)  \nabla_\G  d(\delta_{\sqrt{R}}(x'))   \right) \right) ^{\frac{p}{p-1}} dx' \\
		& \leq &  C \left( \frac{1}{\ln{\sqrt{R}}}\right)^{\frac{p}{p-1}}  R^{\frac{Q}{2} +  \frac{\alpha(p+1)}{2(p-1)}   - \frac{p (\alpha+1)}{p-1}  }  \int\limits_{1 \leq d(x') \leq \sqrt{R}} d(x')^{\alpha\left( \frac{p+1}{p-1}\right)} \left|\Phi\left( \frac{\ln(d(x'))}{\ln(\sqrt{R})} \right) \right|^{\frac{-1}{p-1}} \times \\
		&&   \ \ \ \ \ \ \ \ \ \ \  \ \ \ \ \ \ \  \  \ \ \ \ \ \ \ \ \times   \left(  \nabla_\G \left(d(x')^{-2\alpha - 1}  \Phi'\left( \frac{\ln(d(x'))}{\ln(\sqrt{R})} \right)  \nabla_\G  d(x')   \right) \right) ^{\frac{p}{p-1}} dx' \\
& \leq & C	R^{\frac{-p(\alpha + 1)}{p-1} + \frac{Q}{2} +  \frac{\alpha(p+1)}{2(p-1)}   - \frac{p (\alpha+1)}{p-1}}  \left( \frac{1}{\ln{\sqrt{R}}}\right)^{\frac{p}{p-1}}    \int\limits_{1 \leq d(x') \leq \sqrt{R}}  \left|\frac{(\Phi')^p\left( \frac{\ln(d(x'))}{\ln(\sqrt{R})} \right) + (\Phi'')^p\left( \frac{\ln(d(x'))}{\ln(\sqrt{R})} \right)}{\Phi\left( \frac{\ln(d(x'))}{\ln(\sqrt{R})} \right)} \right|^{\frac{1}{p-1}} dx'.
\end{eqnarray*}
Taking $R \to \infty$ leads to the desired conclusion since $\frac{-p(\alpha + 1)}{p-1} + \frac{Q}{2} +  \frac{\alpha(p+1)}{2(p-1)}   - \frac{p (\alpha+1)}{p-1} < 0$.
\end{proof}

\section{Conjectures}\label{Conjecture}
We present the following two conjectures:
	
\noindent{\bf Conjecture 1.} The problem (\ref{MainProblem}) admits a local solution whenever $1 < p < 1 + \frac{2}{\alpha}$.
	
\noindent{\bf Conjecture 2.} The problem (\ref{MainProblem}) admits a global solution whenever $\frac{Q-\alpha}{Q-2-\alpha} < p < 1+ \frac{2}{\alpha}$.

To support these conjectures, the case in which $|\nabla_\G d| = 1$ almost everywhere (using the $\delta_\lambda$-homogeneity of $d$ this is equivalent to saying that  $|\nabla_\G d| = 1$ almost everywhere on the unit circle with arc length measure) is discussed.
This is true, for example, when $d$ is either a cc-distance \cite[Section 5.2]{Bonfiglioli} (instead of the $\Delta_\G$-gauge) or $\G = (\R^n, +)$ \cite[p. 466]{Bonfiglioli}.
The reader may also refer \cite{MontiSerra} to confirm that such an assumption is not artificial.
The known classical methods for proving the local and global existence work well in this case because the $p$-integrable functions are defined almost everywhere.
Our proofs are written in a way which makes it easier for the reader to observe that the major obstacle while considering $d$ to be the $\Delta_\G$-gauge is that  the infimum of the function $|\nabla_\G d|$ on any given sphere may be zero.
However, if this infimum is not zero, then the proofs work, after  elementary modifications.
Unfortunately, we have been unable to find a proof without imposing this assumption.
	
\subsection{Local existence}
The operator $\Delta_\G$ is infinitesimal generator of a contractive and positive $C_0$-semigroup, which is denoted by $\{ e^{-t \Delta_\G} \}$.
It is easy to verify that a perturbation of $\Delta_\G$ by a bounded operator $V$ will be the infinitesimal generator of a semigroup which we denote by $\{ e^{-t (\Delta_\G + V)} \}$.
	Moreover, if $V$ is bounded by $c$ then $e^{-t (\Delta_\G + V)} \leq e^{-t c} e^{-t \Delta_\G}$ pointwise.
	
It is known from \cite{SuraganTalwar} that (\ref{MainProblem}) without the potential term has a local solution for all $p > 1$.
The proof of \cite[Theorem 2.2 (1)]{SuraganTalwar} may be modified to conclude the existence of a local non-negative solution to \begin{equation*}
		\begin{cases} 
			{u_n}_t(t,x) - \Delta_\G u_n(t,x) - \frac{\lambda \left| \nabla_\G d(x) \right|^2 u_n}{d(x)^2 + (1/n)} =  u_n(t,x)^p + f(x), & (t,x) \in (0,T) \times \G, \\
			u_n(0, x) = u_0(x), & x \in \G, 
		\end{cases}
	\end{equation*} satisfying the integral equation $$u_n(t) = e^{-t(\Delta_\G + V_n)}u_0 +  \int_{0}^{t} e^{-(t-s)(\Delta_\G + V_n)} \left( u_n(s)^p + f(x) \right) ds,$$  where $V_n := \frac{\lambda \left| \nabla_\G d(x) \right|^2 }{(1/n) + d(x)^2}$.
	Here we have used the fact that multiplication by the essentially bounded function $V_n$ defines a bounded operator on $L^q(\G)$ for every $1 \leq q \leq \infty$ -see  \cite[Chapter 3]{Pazy}.
	Since $\{ V_n \}$ is increasing, so is $\{u_n\}$.
	Let the truncation operator $T_n$ be defined by \begin{equation*} T_n(s) = 
		\begin{cases} 
			s, & |s| \leq n, \\
			n \ \text{sign}(s), & |s| \geq n.
		\end{cases}
	\end{equation*}

\begin{theorem}\label{Case1}
Let $1 < p < 1+ \frac{2}{\alpha}$, $\alpha < \beta < \alpha_+(\lambda)$ satisfies $\beta + \alpha < Q$ and suppose that $|\nabla_\G d| = 1$ almost everywhere on $\G$.
Then there exists $\theta, T > 0$  such that for $u_0 \leq a T^{-\theta} d(x)^{-\beta}$ and a suitable choice of small constant $\gamma$ depending upon $T$ and $\theta$ for which $f \leq \gamma d(x)^{-\beta}$, system (\ref{MainProblem}) has a local solution in $L^2(0,T; S_0^1(\G))$.
\end{theorem}
\begin{proof}
Let us start by verifying that $v = a(T - t)^{-\theta} d(x)^{- \beta}$ is a local supersolution to (\ref{MainProblem}) for a suitable choice of $T, \theta, a > 0$.
That is, there exists $T,\theta, a > 0$ for which $$v_t(t,x) - \Delta_\G v(t,x) - \frac{\lambda \left| \nabla_\G d(x) \right|^2 v(t,x)}{d(x)^2} \geq  v(t,x)^p +  f(x).$$
In view of the assumption on $f$, it suffices to prove that 
		$$v_t(t,x) - \Delta_\G v(t,x) - \frac{\lambda \left| \nabla_\G d(x) \right|^2 v(t,x)}{d(x)^2} \geq  v(t,x)^p + \gamma d(x)^{-\beta }.$$
		Now,
		$$v^p (t,x) = a^p(T - t)^{-p \theta} d(x)^{- p \beta},$$
		$$v_t(t,x) = a \theta (T - t)^{-\theta - 1} d(x)^{- \beta}$$ and
		$$\Delta_\G v(t,x)= a(T-t)^{-\theta} |\nabla_\G d(x)|^2 d^{- \beta - 2}(\beta^2 - (Q-2) \beta).$$
		Thus,
		\begin{eqnarray*}
			&& v_t(t,x) - \Delta_\G v(t,x) - \frac{\lambda \left| \nabla_\G d(x) \right|^2 v(t,x)}{d(x)^2} \\
			&=& a\left( \theta (T - t)^{-\theta - 1} d(x)^{- \beta} - (T-t)^{-\theta} |\nabla_\G d(x)|^2 d(x)^{- \beta - 2}(\beta^2 - (Q-2) \beta) \right) \\
			& - & a \left( \frac{\lambda \left| \nabla_\G d(x) \right|^2 (T - t)^{-\theta} d(x)^{- \beta}}{d(x)^2} \right) \\
			& = & a \left( \theta (T - t)^{-\theta - 1} d(x)^{- \beta} - (T-t)^{-\theta} |\nabla_\G d(x)|^2 d(x)^{- \beta - 2}(\beta^2 - (Q-2) \beta + \lambda) \right) \\
			&=&  a (T - t)^{-\theta - 1} d(x)^{-\beta - 2} \left( \theta  d(x)^{2}   -  (T-t) |\nabla_\G d(x)|^2 (\beta^2 - (Q-2) \beta + \lambda)  \right) .
		\end{eqnarray*}
		
		From the assumption that $\alpha < \beta < \alpha_+(\lambda)$, we have $\kappa := - \beta^2 + (Q-2) \beta  - \lambda > 0$.
		Also, $|\nabla_\G d| = 1$ almost everywhere on $\G$.
		It thus suffices to prove that 
		$$ \theta  d(x)^{2}   + \kappa  (T-t)  \geq a^{p-1}(T-t)^{-p \theta  + \theta + 1 } d(x)^{- p \beta + \beta + 2} +  \frac{\gamma(T,\theta)}{a}(T-t)^{\theta + 1 } d(x)^2,$$ for some $\theta, T, a, \gamma(T, \theta) > 0$.
Using the facts that $\beta > 0$ and $1 < p < 1+ \frac{2}{\alpha}$, we obtain $$\beta < p \beta < \beta + \frac{2 \beta}{\alpha} < \beta + 2.$$
		Also, for $$\theta \leq \frac{1}{p-1},$$ we have $$- p \theta + \theta + 1 \geq 0.$$
		In view of these inequalities, large $\theta$ and small $T, a$ and $\gamma$ may be chosen so that the required inequality is satisfied.
		This shows that $v$ is a supersolution to (\ref{MainProblem}).
		Let us denote this supersolution by $\ol{u}$.
		
		Consider now the problem \begin{equation*}
			\begin{cases} 
				(v_0)_t(t,x) - \Delta_\G v_0(t,x) = f(x)  & (t,x) \in (0,T) \times B_0(1), \\
				v_0(0, x) = T_1(\ol{u}(0,x)), & x \in B_0(1),\\
				v_0(t,x) = 0, & (t,x) \in (0,T) \times \partial B_0(1).
			\end{cases}
		\end{equation*}
		From Duhamel's principle, there exists a solution of this heat equation in $$L^2([0,T), S_0^1(B_0(1))).$$
		From comparison theorem \cite[Theorem 2.1]{RuzhanskySuraganBLMS} we obtain that $v_0 \leq \ol{u}$.
		Using  $v_0$ as a starting point, and denoting again by $v_0$ the extention of $v_0$ to $B_0(2)$ given by zero outside $B_0(1)$, we recursively construct a sequence of solutions $\{v_n \in L^2([0,T), S_0^1(B_0(n+1))) \}$ to the problems
		\begin{equation*}
			\begin{cases} 
				(v_n)_t(t,x) - \Delta_\G v_n(t,x) = f(x) + \frac{\lambda \left| \nabla_\G d \right|^2 v_{n-1}(t,x)}{1/n + d(x)^2} + v_{n-1}^p(t,x) & (t,x) \in (0,T) \times B_0(n+1), \\
				v_0(0, x) = T_n(\ol{u}(0, x)), & x \in B_0(n+1),\\
				v_0(t,x) = 0, & (t,x) \in (0,T) \times \partial B_0(n+1).
			\end{cases}
		\end{equation*}
		Such a sequence exists because on the right hand side of every problem there  is a fixed function.
		Moreover, from comparison principle it follows that the series $\{v_n\}$ of non-negative increasing functions is bounded above by $\ol{u}$.
		Let $w$ denote the pointwise limit of a subsequence of $v_n$ in  $L^2( (0,T), L^2_{loc}(\G))$.
		As in \cite{GoldsteinZhang, BarasGoldstein, Prignet}, it follows from standard techniques that $w \in L^2((0,T), S^1_0(\G))$ is a weak solution of (\ref{MainProblem}).
	\end{proof}
	
	\subsection{Global existence}
	Whenever $p > \frac{Q-\alpha}{Q-2-\alpha}$, we anticipate the existence of $u_0$ and $f$ for which a global weak solution of (\ref{MainProblem}) exists.
	Moreover, the initial condition is expected to  depend on the forcing term in the sense that $u_0$ may be chosen to be a scalar multiple of $f$.
	Although we could not prove the same, we now discuss the following problem in which $f$ is dependent on time variable as well.
	Consider,
	\begin{equation}\label{MainProblem3}
		\begin{cases} 
			u_t(t,x) - \Delta_\G u(t,x) - \frac{\lambda \left| \nabla_\G d(x) \right|^2 u}{d(x)^2} =  u(t,x)^p + f(t, x), & (t,x) \in (0,T) \times \G, \\
			u(0, x) = u_0(x), & x \in \G
		\end{cases}
	\end{equation}
	where $\lambda > 0$ and $f,u \geq 0$.
It is worth noting that in the absence of the function $f$, the approach outlined in \cite{AbdellaouiPearlPrimo} may be used to demonstrate the global existence when $p > \frac{Q+2-\alpha}{Q-\alpha}$.
However, when $f \neq 0$ and is time dependent, determining a precise critical exponent becomes challenging.
Despite this, we have established the following results for this scenario:
	
	\begin{theorem}\label{ExistenceOfFamilyOfSupersolutions}
		Let $|\nabla_\G d| = 1$ almost everywhere on $\G$ and  $\frac{Q-\alpha}{Q-2-\alpha} < p < 1+ \frac{2}{\alpha}$  in (\ref{MainProblem3}).
		Then there exists constants $A, c, \gamma$ such that for $$f(t,r)  \leq  A (T+t)^{\frac{-p}{p-1}}\left(\frac{cr}{(T+t)^{\frac{1}{2}}}\right)^{-\gamma} \exp^{\frac{(cr)^2}{4(T+t)}},$$ the problem (\ref{MainProblem3}) has a family of radial global supersolution.
	\end{theorem}
	\begin{proof}
		To find a family of radial supersolutions, note that the radial form of $$u_t(t,x) - \Delta_\G u(t,x) - \frac{\lambda \left| \nabla_\G d(x) \right|^2 u}{d(x)^2}  \geq u(t,x)^p + f(t, d(x))$$ is
		$$u_t(t,r) - |\nabla_\G r|^2 \left( u_{rr}(t,r) + \frac{Q-1}{r}u_r(t,r) + \frac{\lambda  u(t, r)}{r^2} 	\right)    \geq  u(t,r)^p + f(t, r).$$
		For the smooth bounded and positive function $$g(s):= A (cs)^{-\gamma} \exp^{\frac{- (cs)^2}{4}},$$ and any $T > 0$, define $$u(r,t,T):= (T+t)^{\frac{-1}{p-1}} g \left( \frac{r}{(T+t)^{\frac{1}{2}}} \right).$$
		Then, for $$s = \frac{r}{(T+t)^{\frac{1}{2}}},$$ we have $$u_t =  -(T+t)^{\frac{-p}{p-1}} \left( \frac{g}{p-1} +  \frac{r g'}{2 (T+t)^{\frac{1}{2}}} \right),$$
		$$ u_r = (T+t)^{\frac{-(p+1)}{2(p-1)}} g'(s) ,$$
		$$u_{rr} = (T+t)^{\frac{-p}{p-1}} g''(s),$$
		where 
		$$g' = -A c^{- \gamma} s^{-\gamma - 1} \exp^{\frac{- (cs)^2}{4}} \left(  \frac{c^2 s^2}{2} + \gamma  \right),$$
		and
		$$g''(s)= -A c^{-\gamma} \exp^{\frac{- (cs)^2}{4}} \left(  \frac{c^2 s^{-\gamma}}{2} - c^2 \gamma s^{-\gamma} -\frac{c^4 s^{-\gamma + 2} }{4} -\gamma^2 s^{-\gamma - 2} - \gamma s^{- \gamma - 2} \right) .$$
		For convenience, set $$l = \frac{(cs)^2}{4}.$$
		Now, we need to choose positive constants $A, c$ and $\gamma$ such that 
		
		\begin{eqnarray*}
			0 &\geq & u(t,r)^p + f(t, r) - u_t(t,r) + |\nabla_\G r|^2 \left( u_{rr}(t,r) + \frac{Q-1}{r}u_r(t,r) + \frac{\lambda  u(t, r)}{r^2} 	\right)  \\
			&=& (T+t)^{\frac{-p}{p-1}} A^p (cs)^{-p\gamma} \exp^{-lp} + f(t,r) \\
			&+&  (T+t)^{\frac{-p}{p-1}} \left( \frac{A (cs)^{-\gamma} \exp^{-l}}{p-1} +  \frac{ -A r c^{- \gamma} s^{-\gamma - 1} \exp^{-l} \left(  \frac{c^2 s^2}{2} + \gamma  \right) }{2 (T+t)^{\frac{1}{2}}} \right) \\
			&+& |\nabla_\G r|^2 \left( (T+t)^{\frac{-p}{p-1}} \left( -A c^{-\gamma} \exp^{-l} \left(  \frac{c^2 s^{-\gamma}}{2} - c^2 \gamma s^{-\gamma} -\frac{c^4 s^{-\gamma + 2} }{4} - (\gamma^2 + \gamma) s^{-\gamma - 2}  \right) \right) \right) \\
			& +& |\nabla_\G r|^2 \left( \frac{Q-1}{r} \left( (T+t)^{\frac{-(p+1)}{2(p-1)}} \right) \left( -A c^{- \gamma} s^{-\gamma - 1} \exp^{-l} \left(  \frac{c^2 s^2}{2} + \gamma  \right)   \right) 	\right)\\
			&+& |\nabla_\G r|^2 \left(  \frac{\lambda  (T+t)^{\frac{-1}{p-1}}   A (cs)^{-\gamma} \exp^{-l}}{r^2}  \right).
		\end{eqnarray*}
		This is equivalent to proving the existence of $A,c$ and $\gamma$ for which
		\begin{eqnarray*}
			0& \geq & (T+t)^{\frac{-p}{p-1}} A^{p-1} (cs)^{(1-p)\gamma} \exp^{(1-p)l} + (cs)^\gamma A^{-1}f(t,r) \exp^{l} \\
			&+&  (T+t)^{\frac{-p}{p-1}} \left( \frac{  1}{p-1} +  \frac{ - r  s^{- 1}  \left(  \frac{c^2 s^2}{2} + \gamma  \right) }{2 (T+t)^{\frac{1}{2}}} \right) \\
			&+& |\nabla_\G r|^2 \left( (T+t)^{\frac{-p}{p-1}} \left( -   \left(  \frac{c^2 }{2} - c^2 \gamma  -\frac{c^4 s^{ 2} }{4} -\gamma^2 s^{- 2} - \gamma s^{- 2} \right) \right) \right) \\
			& +& |\nabla_\G r|^2 \left( \frac{Q-1}{r} \left( (T+t)^{\frac{-(p+1)}{2(p-1)}} \right) \left( -  s^{- 1}  \left(  \frac{c^2 s^2}{2} + \gamma  \right)   \right)	+  \frac{\lambda  (T+t)^{\frac{-1}{p-1}}  }{r^2}  \right)\\
			&=& (T+t)^{\frac{-p}{p-1}} A^{p-1} (cs)^{(1-p)\gamma} \exp^{(1-p)l} + (cs)^\gamma A^{-1}f(t, r) \exp^{l} \\
			&+&  (T+t)^{\frac{-p}{p-1}} \left( \frac{  1}{p-1} - \frac{     \frac{c^2 s^2}{2} + \gamma  }{2} \right) \\
			&+& |\nabla_\G r|^2 \left( (T+t)^{\frac{-p}{p-1}} \left( -   \left(  \frac{c^2 }{2} - c^2 \gamma  -\frac{c^4 s^{ 2} }{4} -\gamma^2 s^{- 2} - \gamma s^{- 2} \right) \right) \right) \\
			& +& |\nabla_\G r|^2 \left( \frac{Q-1}{r} \left( (T+t)^{\frac{-(p+1)}{2(p-1)}} \right) \left( -  s^{- 1}  \left(  \frac{c^2 s^2}{2} + \gamma  \right)   \right)	+  \frac{\lambda  (T+t)^{\frac{-p}{p-1}}  }{s^2}  \right)\\
			&=& (T+t)^{\frac{-p}{p-1}} A^{p-1} (cs)^{(1-p)\gamma} \exp^{(1-p)l} + (cs)^\gamma A^{-1}f(t,r) \exp^{l} \\
			&+&  (T+t)^{\frac{-p}{p-1}} \left( \frac{  1}{p-1} - \frac{     \frac{c^2 s^2}{2} + \gamma  }{2} \right) \\
			&+& |\nabla_\G r|^2 \left( (T+t)^{\frac{-p}{p-1}} \left( -   \left(  \frac{c^2 }{2} - c^2 \gamma  -\frac{c^4 s^{ 2} }{4} -\gamma^2 s^{- 2} - \gamma s^{- 2} \right) \right) \right) \\
			& +& |\nabla_\G r|^2 \left( \frac{Q-1}{s} \left( (T+t)^{\frac{-p}{p-1}} \right) \left( -  s^{- 1}  \left(  \frac{c^2 s^2}{2} + \gamma  \right)   \right)	+  \frac{\lambda  (T+t)^{\frac{-p}{p-1}}  }{s^2}  \right).
		\end{eqnarray*}
		Thus, we need to prove that
		\begin{eqnarray*}
			0& \geq & A^{p-1} (cs)^{(1-p)\gamma} \exp^{(1-p)l} + (T+t)^{\frac{p}{p-1}}(cs)^\gamma A^{-1}f(t,r) \exp^{l} \\
			&+&   \left( \frac{  1}{p-1} - \frac{c^2 s^2}{4} - \frac{ \gamma  }{2} \right) \\
			&-& |\nabla_\G r|^2    \left(  \frac{c^2 }{2} - c^2 \gamma  -\frac{c^4 s^{ 2} }{4} -\gamma^2 s^{- 2} - \gamma s^{- 2} +  \frac{Q-1}{s^2}  \left(  \frac{c^2 s^2}{2} + \gamma  \right)  	-  \frac{\lambda    }{s^2}   \right)\\
			&=& A^{p-1} (cs)^{(1-p)\gamma} \exp^{(1-p)l} + (T+t)^{\frac{p}{p-1}}(cs)^\gamma A^{-1}f(t,r) \exp^{l} \\
			&+&   \left( \frac{  1}{p-1} - \frac{c^2 s^2}{4} - \frac{ \gamma  }{2} \right) \\
			&-& |\nabla_\G r|^2    \left(  \frac{c^2 }{2} - c^2 \gamma  -\frac{c^4 s^{ 2} }{4} - \frac{(\gamma^2 -Q \gamma + 2 \gamma  + \lambda   )}{s^2}  +  \frac{(Q-1) c^2}{2}      \right)
		\end{eqnarray*}
		for some $A, c$ and $\gamma$.

		In view of the hypothesis that $$(T+t)^{\frac{p}{p-1}}(cs)^\gamma A^{-1}f(t,r) \exp^{l} \leq 1,$$ it suffices find $A,c$ and $\gamma$ satisfying 
		\begin{eqnarray*}
			0& \geq & A^{p-1} (cs)^{(1-p)\gamma} \exp^{(1-p)l} +  1 +  \frac{  1}{p-1} - \frac{c^2 s^2}{4} - \frac{ \gamma  }{2} \\
			&+& |\nabla_\G r|^2    \left( \frac{c^4 s^{ 2} }{4} + \frac{(\gamma^2 -Q \gamma + 2 \gamma  + \lambda   )}{s^2}  +  c^2 \left( \gamma -\frac{Q}{2}  \right) \right).
		\end{eqnarray*}
		Note that for any $\frac{2}{p-1} >\gamma > \alpha$, we have $$\gamma^2 -Q \gamma + 2 \gamma  + \lambda  < 0.$$	
		Since $|\nabla_\G r| \equiv 1$, we need to prove that 
		\begin{eqnarray*}
			0& \geq & A^{p-1} (cs)^{(1-p)\gamma} \exp^{(1-p)l} +  1 +  \frac{  1}{p-1} - \frac{c^2 s^2}{4} - \frac{ \gamma  }{2} \\
			&+&  \frac{ c^4 s^{ 2} }{4} + \frac{ \gamma^2 -Q \gamma + 2 \gamma  + \lambda   }{s^2}  +  c^2 \left( \gamma - \frac{Q}{2}  \right)\\
			&=& A^{p-1} (cs)^{(1-p)\gamma} \exp^{(1-p)l} + \frac{\gamma^2 -Q \gamma + 2 \gamma  + \lambda   }{s^2}  + \frac{c^4 s^{ 2} }{4} - \frac{c^2 s^2}{4} \\
			&+&  c^2 \left( \gamma -\frac{Q}{2}  \right) - \frac{ \gamma  }{2} +  1+  \frac{  1}{p-1}.
		\end{eqnarray*}
		As $\gamma < \frac{2}{p-1}$,  one may choose small $A> 0$ such that $$\frac{\gamma^2 -Q \gamma + 2 \gamma  + \lambda   }{s^2} + A^{p-1} (cs)^{(1-p)\gamma} \exp^{\frac{(cs)^2 (1-p)}{4}} \leq 0.$$
		Now, let $$H(c,z):=  c^2   \left( z -\frac{Q}{2}  \right) - \frac{z}{2} +  1 +  \frac{  1}{p-1}.$$
		Since $\frac{Q-\alpha}{Q-2-\alpha} < p$, we obtain
		$$H\left( 1, \alpha \right) =   \left( \alpha -\frac{Q}{2}  \right) - \frac{\alpha}{2} +  1 +  \frac{1}{p-1} =   \frac{\alpha}{2} -\frac{Q}{2}    +  1 +  \frac{1}{p-1}  < 0.$$
		Using continuity of $H$, one may choose $c <  1$ and $\alpha< \gamma < \frac{2}{p-1}$ such that $$H(c, \gamma) <0.$$
		Moreover, as $c < 1$, we have 
		$$\frac{ c^4 s^{ 2} }{4} - \frac{c^2 s^2}{4} = \frac{c^2s^2}{4}(c^2 - 1) < 0.$$
		This proves the existence of $A, c$ and $\gamma$.
		So, $$u(r,t,T):= A (rc)^{- \gamma}  (T+t)^{\frac{\gamma}{2} - \frac{1}{p-1}} \exp^{\frac{- \left(c r \right)^2}{4 (T+t)}}$$
		is a family of supersolutions for (\ref{MainProblem3}).
	\end{proof}

	\begin{theorem}
		Let $|\nabla_\G d| = 1$ almost everywhere on $\G$ and  $\frac{Q-\alpha}{Q-2-\alpha} < p < 1+ \frac{2}{\alpha}$ in (\ref{MainProblem3}).
		Then there exists some $f(t,x)$ and $u_0(x)$ such that the problem (\ref{MainProblem3}) has a global solution.
	\end{theorem}
	\begin{proof}
		Using the family of radial supersolutions found in \Cref{ExistenceOfFamilyOfSupersolutions}, one may obtain a supersolution $w \in L^2((0,T), S_0^1(\G))$ for every $T >0$.
		Then by an iteration argument, as in \Cref{Case1}, one obtains the existence of a global solution  $u \in L^2((0,T), S_0^1(\G))$ for every $T >0$, which  is bounded above by $w$.
	\end{proof}

%	\section*{Statements and declarations}
	
%	\subsection*{Funding}
%This research was funded by Nazarbayev University under grants 20122022CRP1601 and 20122022FD4105.
%Talwar gratefully acknowledges the support of the Science and Engineering Research Board (SERB) through the National Post-Doctoral Fellowship (NPDF), Grant No. PDF/2023/000688. This support was instrumental in the completion of a portion of this manuscript.

\subsection*{Acknowledgments}
We thank the anonymous reviewers for carefully reading the earlier version of this manuscript and suggesting several changes which helped us correct certain inequalities thereby improving the article.
%Talwar gratefully acknowledges the support of the Science and Engineering Research Board (SERB) through the National Post-Doctoral Fellowship (NPDF), Grant No. PDF/2023/000688. This support was instrumental in the completion of a portion of this manuscript.
	
%	\subsection*{Competing interests}
%	The authors have no relevant financial or non-financial interests to disclose.
	
%	\subsection*{Author contributions}
%	This research is a result of continuous discussions between the authors and every part of it is contributed by both of them.
	
%	\subsection*{Data availability}
%	Not applicable.


\begin{thebibliography}{99}
		\bibitem{AbdellaouiAlonsoRamos} B. Abdellaoui, I. P. Alonso and A.~P. Ramos, Influence of the Hardy potential in a semilinear heat equation, Proc. Roy. Soc. Edinburgh Sect. A {\bf 139} (2009), no.~5, 897--926. 
		\bibitem{AbdellaouiMedinaPearlPrimo} B. Abdellaoui,, M. Medina, I. Peral and A. Primo, Optimal results for the fractional heat equation involving the Hardy potential, Nonlinear Analysis 140 (2016), 166-207.
		\bibitem{AbdellaouiPearlPrimo}B. Abdellaoui, I. Peral and A. Primo, A note on the Fujita exponent in fractional heat equation involving the Hardy potential, arXiv:1911.07578 (2019).
		\bibitem{AbdellaouiSiclariPrimo} B. Abdellaoui, G. Siclari and A. Primo, Fujita exponent for non-local parabolic equation involving the Hardy–Leray potential, Journal of Evolution Equations 24.3 (2024): 55.
		\bibitem{Ambrosio} L. D'Ambrosio, Hardy-type inequalities related to degenerate elliptic differential operators, Ann. Sc. Norm. Super. Pisa Cl. Sci. (5) {\bf 4} (2005), no.~3, 451--486.
		\bibitem{AmbrosioMitidieri} L. D'Ambrosio and E.~Mitidieri,
		A priori estimates and reduction principles for quasilinear elliptic problems and applications. Advances in Differential Equations, 17(2012), 935–1000.
		\bibitem{AvelinCapognaCittiNysrtom} B. Avelin,  L Capogna, G Citti and K Nystr\"{o}m, Harnack estimates for degenerate parabolic equations modeled on the subelliptic $p$-Laplacian, Adv. Math. {\bf 257} (2014), 25--65.
		\bibitem{Bandle} C. Bandle, H. A. Levine\ and\ Q. S. Zhang, Critical exponents of Fujita type for inhomogeneous parabolic equations and systems, J. Math. Anal. Appl. {\bf 251} (2000), no.~2, 624--648.
		\bibitem{BarasGoldstein} P. Baras\ and\ J.~A. Goldstein, The heat equation with a singular potential, Trans. Amer. Math. Soc. {\bf 284} (1984), no.~1, 121--139.
		\bibitem{Bonfiglioli} A. Bonfiglioli, E. Lanconelli\ and\ F. Uguzzoni, Stratified Lie groups and potential theory for their sub-Laplacians, Springer Monographs in Mathematics, Springer, Berlin, 2007.
		\bibitem{BaramantiBook} M. Bramanti, {\it An invitation to hypoelliptic operators and H\"{o}rmander's vector fields}, SpringerBriefs in Mathematics, Springer, Cham, 2014.
		\bibitem{Folland} G. B. Folland, Subelliptic estimates and function spaces on nilpotent Lie groups, Ark. Mat. {\bf 13} (1975), no.~2, 161--207.
		\bibitem{Fujita} H. Fujita, On the blowing up of solutions of the Cauchy problem for $u\sb{t}=\Delta u+u\sp{1+\alpha }$. J. Fac. Sci. Univ. Tokyo Sect. I 13 (1966), 109--124 (1966).
		\bibitem{GoldsteinGoldsteinKogojRhandiTacelli} G.~R. Goldstein, J.~A. Goldstein, A.~E. Kogoj, A. Rhandi and C. Tacelli, Instantaneous blowup and singular potentials on Heisenberg groups. Ann. Sc. Norm. Super. Pisa Cl. Sci. (5) 23 (2022), no. 4, 1723--1748.
		\bibitem{GoldsteinKombe} J. A. Goldstein\ and\ I. K\"{o}mbe, The Hardy inequality and nonlinear parabolic equations on Carnot groups, Nonlinear Anal. {\bf 69} (2008), no.~12, 4643--4653.
		\bibitem{GoldsteinZhang} J. A. Goldstein\ and\ Q. S. Zhang, On a degenerate heat equation with a singular potential, J. Funct. Anal. {\bf 186} (2001), no.~2, 342--359.
		\bibitem{Gu} Q. Gu, Y.Sun, J.Xiao and F. Xu, Global positive solution to a semi-linear parabolic equation with potential on Riemannian manifold, Calc. Var. Partial Differential Equations {\bf 59} (2020), no.~5, Paper No. 170, 24 pp.
		\bibitem{Hayakawa} K. Hayakawa, On nonexistence of global solutions of some semilinear parabolic differential equations, Proc. Japan Acad. {\bf 49} (1973), 503--505.
		\bibitem{Ishige} K. Ishige, On the Fujita exponent for a semilinear heat equation with a potential term, J. Math. Anal. Appl. {\bf 344} (2008), no.~1, 231--237.
		\bibitem{MontiSerra} R. Monti\ and\ F. Serra Cassano, Surface measures in Carnot-Carath\'{e}odory spaces, Calc. Var. Partial Differential Equations {\bf 13} (2001), no.~3, 339--376.
		\bibitem{Pazy} A. Pazy, {\it Semigroups of linear operators and applications to partial differential equations}, Applied Mathematical Sciences, 44, Springer, New York, 1983.
		\bibitem{Prignet} A.~Prignet, Existence and uniqueness of ``entropy'' solutions of parabolic problems with $L^1$ data, Nonlinear Anal. {\bf 28} (1997), no.~12, 1943--1954.
		\bibitem{RothschildStein} L.~P. Rothschild\ and\ E.~M. Stein, Hypoelliptic differential operators and nilpotent groups, Acta Math. {\bf 137} (1976), no.~3-4, 247--320.
		\bibitem{RuzhanskySuraganBLMS} M.~V. Ruzhansky\ and\ D. Suragan, A comparison principle for nonlinear heat Rockland operators on graded groups, Bull. Lond. Math. Soc. {\bf 50} (2018), no.~5, 753--758.
		\bibitem{RuzhanskySuraganPotentialAnalysis}M.~V. Ruzhansky\ and\ D. Suragan, Green's identities, comparison principle and uniqueness of positive solutions for nonlinear $p$-sub-Laplacian equations on stratified Lie groups, Potential Anal. {\bf 53} (2020), no.~2, 645--658.
		\bibitem{Suragan} D. Suragan, A survey of Hardy type inequalities on homogeneous groups, in {\it Mathematical analysis, its applications and computation}, 99--122, Springer Proc. Math. Stat., 385, Springer, Cham.
		\bibitem{SuraganTalwar2} D. Suragan and B. Talwar, Fujita exponent on stratified Lie groups, Collect. Math. (2023). https://doi.org/10.1007/s13348-023-00427-3
		\bibitem{SuraganTalwar} D. Suragan and B. Talwar, Nonexistence of solutions of certain semilinear heat equations,  arXiv:2210.11307 (2022).
	\end{thebibliography}
\end{document}